\documentclass[final,onefignum,onetabnum]{siamart190516}

\ifpdf
\hypersetup{
  pdftitle={},
  pdfauthor={Y. Kazashi, Y. Suzuki, and T. Goda}
}
\fi

\usepackage{amsfonts}
\usepackage{amssymb}    \usepackage{amsfonts,mathrsfs}
\usepackage{dsfont} \usepackage{graphicx}
\usepackage{epstopdf}
\usepackage{algorithmic}
\ifpdf
  \DeclareGraphicsExtensions{.eps,.pdf,.png,.jpg}
\else
  \DeclareGraphicsExtensions{.eps}
\fi

\newtheorem{Definition}{Definition}[section]
\newsiamremark{remark}{Remark}
\newsiamremark{hypothesis}{Hypothesis}
\crefname{hypothesis}{Hypothesis}{Hypotheses}
\newsiamthm{claim}{Claim}

\headers{Sub-optimality of Gauss--Hermite and optimality of trapezoidal}{Y. Kazashi, Y. Suzuki, and T. Goda}

\title{Sub-optimality of Gauss--Hermite quadrature and optimality of the trapezoidal rule for\\ functions with finite smoothness\thanks{Submitted to the editors DATE.
\funding{This work of the second author was supported by NTNU project grant 81617985.}}}

\author{
    {Yoshihito Kazashi}
    \thanks{Corresponding author. Institute for Applied Mathematics,  Heidelberg University, Im Neuenheimer Feld 205, 69120 Heidelberg, Germany (\email{y.kazashi@uni-heidelberg.de}).}
\and 
    {Yuya Suzuki}
\thanks{Corresponding author. Department of Mathematical Sciences, Norwegian University of Science and Technology, Sentralbygg II,
Alfred Getz' vei 1,
Gl{\o}shaugen,
7034 Trondheim, Norway (\email{yuya.suzuki@ntnu.no}).}
\and
{Takashi Goda}
    \thanks{School of Engineering, University of Tokyo, 7-3-1 Hongo, Bunkyo-ku, Tokyo 113-8656, Japan (\email{goda@frcer.t.u-tokyo.ac.jp}).}
}

\usepackage{amsopn}

\usepackage{float}

\newcommand*{\bbR}{\mathbb{R}}

\newcommand*{\rme}{\mathrm{e}}

\newcommand*{\rd}{\mathrm{d}}

\newcommand*{\calO}{\mathcal{O}}

\newcommand*{\cH}{\mathcal{H}}

\newcommand*{\Hscr}{\mathscr{H}}

\newcommand*{\R}{\bbR}

\newcommand{\wor}{\mathrm{wor}}
\newcommand{\GH}{\mathrm{GH}}

\newcommand{\NN}{\mathbb{N}}
\newcommand{\RR}{\mathbb{R}}
\newcommand{\ZZ}{\mathbb{Z}}

\newcommand{\dsone}{\mathds{1}}

\DeclareMathOperator{\erf}{erf}

\allowdisplaybreaks
\usepackage{xcolor}
\definecolor{darkred}{RGB}{220,20,60} 
\definecolor{darkblue}{RGB}{0,60,180} 
\definecolor{darkgreen}{RGB}{0,130,70}
\definecolor{darkorange}{RGB}{180,60,0}
\newcommand{\yk}[1]{{\color{darkblue}{#1}}} \usepackage[normalem]{ulem}
 
\newcommand{\ys}[1]{{\color{darkorange}{#1}}}

\usepackage[normalem]{ulem}
\newcommand*{\ykcst}[1]			  					            {\relax
        \ifmmode\text{\textcolor{darkblue}{\sout{\ensuremath{#1}}}}
    \else\textcolor{darkblue}{\sout{#1}}\fi}

\usepackage{todonotes,varwidth}
\makeatletter
\tikzstyle{diaanotestyle} = [
    draw=\@todonotes@currentbordercolor,
    fill=\@todonotes@currentbackgroundcolor,
    line width=0.5pt,
    inner sep = 0.8 ex,
    rounded corners=4pt,align=left,
   ]

\renewcommand{\@todonotes@drawInlineNote}{{\begin{tikzpicture}[remember picture,baseline={(0,0)}]\draw node[diaanotestyle,font=\@todonotes@sizecommand,anchor=base west]{\begin{varwidth}[t]{7cm}
                \if@todonotes@authorgiven {\@todonotes@sizecommand \@todonotes@author:\,\@todonotes@text}\else {\@todonotes@sizecommand \@todonotes@text}\fi
                \end{varwidth}};\end{tikzpicture}}}\makeatother

\begin{document}

\maketitle

\begin{abstract}
The sub-optimality of Gauss--Hermite quadrature and the optimality of the trapezoidal rule 
are proved in the weighted Sobolev spaces of square integrable functions of order $\alpha$,
where the optimality is in the sense of worst-case error. 
For Gauss--Hermite quadrature, we obtain matching lower and upper bounds, which turn out to be merely of the order $n^{-\alpha/2}$ with $n$ function evaluations, although the optimal rate for the best possible linear quadrature is known to be $n^{-\alpha}$. 
Our proof of the lower bound exploits the structure of the Gauss--Hermite nodes; the bound is independent of the quadrature weights, and changing the Gauss--Hermite weights cannot improve the rate $n^{-\alpha/2}$.
In contrast, we show that a suitably truncated trapezoidal rule achieves the optimal rate up to a logarithmic factor. 
\end{abstract}  
\begin{keywords}
  Gauss--Hermite quadrature, trapezoidal rule, weighted Sobolev space, worst-case error
\end{keywords}

\begin{AMS}
65D30, 65D32, 65Y20, 33C45
\end{AMS}

\section{Introduction} \label{sec:intro}
This paper is concerned with a sub-optimality of Gauss--Hermite quadrature and an optimality of the trapezoidal rule.

Given a function $f\colon \mathbb{R}\to \mathbb{R}$, Gauss--Hermite quadrature is 
one of the standard numerical integration methods to compute the integral
\begin{equation}
    I(f):=\int_\mathbb{R}f(x)\frac{1}{\sqrt{2\pi}}\mathrm{e}^{-x^2/2}\mathrm{d}x.
    \label{eq:def_I(f)}
\end{equation}
It is a Gauss-type quadrature formula, i.e., the quadrature points are the zeros of the degree $n$ orthogonal polynomial associated with the weight function, and the corresponding quadrature weights are readily defined. With the weight function $\rho(x):=\frac{1}{\sqrt{2\pi}}\mathrm{e}^{-x^2/2}$, the orthogonal polynomial we have is the (so-called probabilist's) Hermite polynomial.

Gauss--Hermite quadrature is widely used; 
here, we just mention 
spectral methods \cite{Guo.BY_1999_Hermite,mao2017hermite,Canuto.C_etal_2006_book}, and applications in 
aerospace engineering \cite{braun2021satellite,BBM2020book}, 
finance \cite{Brandimarte2006book,FR2008book,GH2010}, and 
physics \cite{SS2016book,Gezerlis2020book}.
Nevertheless, the limitation of this method seems to be less known.

We start with numerical results that illustrate this deficiency. 
Figure~\ref{fig:GH-sub} shows a comparison of the Gauss--Hermite rule and a suitably truncated trapezoidal rule on $\mathbb{R}$.  
Here the target function in \eqref{eq:def_I(f)} is $f(x)=|x|^p$ with $p\in \{1,3,5\}$. 
For the trapezoidal rule, we integrate $f(x)\rho(x)$ with a suitable cut-off of the domain. 
We discuss the setting of this experiment in more detail at the end of Section~\ref{sec:trapez}.
\begin{figure}[H]
\centering
 \includegraphics[scale=1.05]{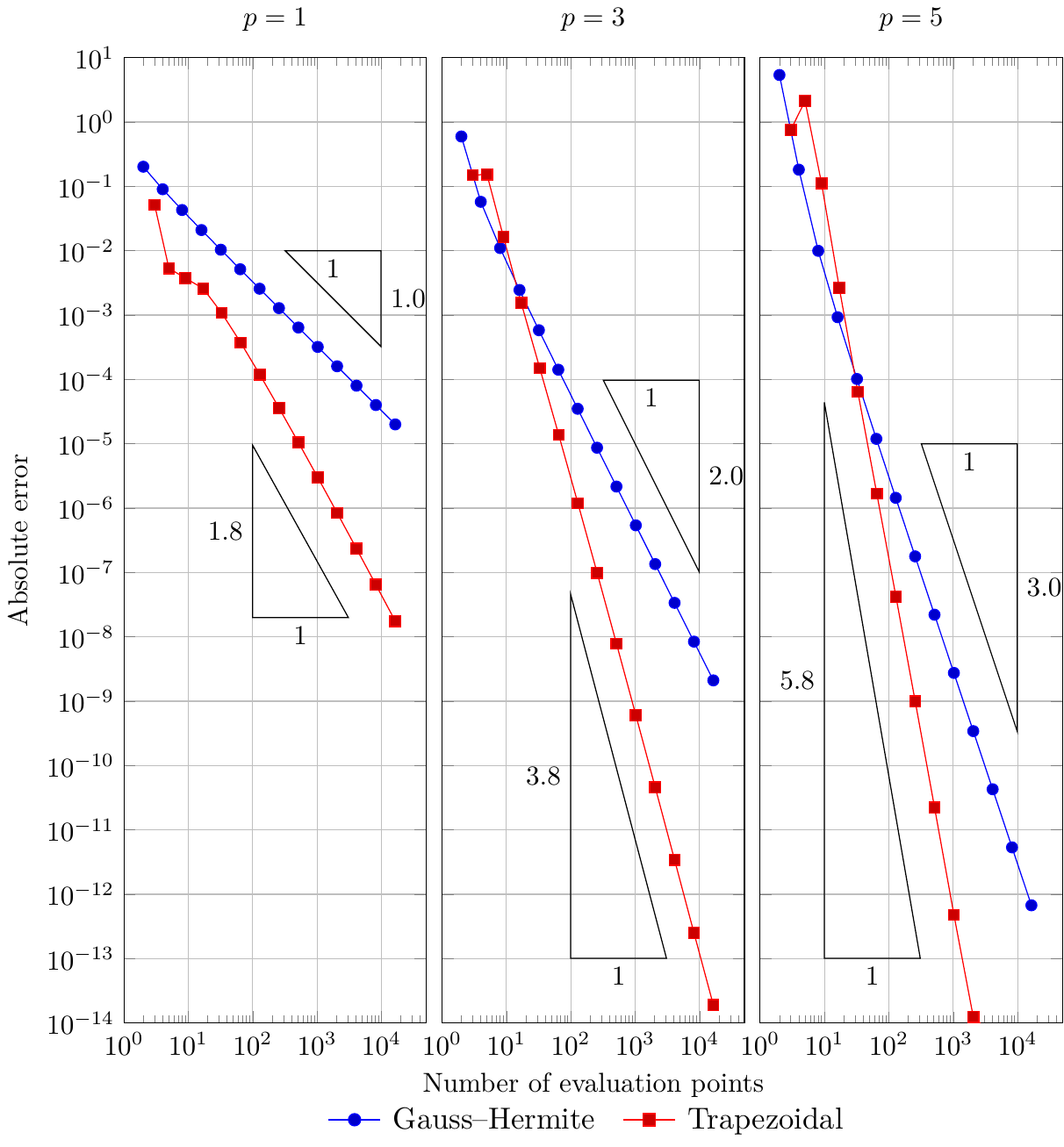}
 \vspace*{-7mm}
\caption{Absolute integration errors for $f(x)=|x|^p$ with $p=1$ \textup{(}left{)}, $p=3$ \textup{(}centre{)}, and $p=5$ \textup{(}right{)}.
Gauss--Hermite rule exhibits a slower error decay than the trapezoidal rule does.}
 \label{fig:GH-sub}
\end{figure}

What we observe in Figure~\ref{fig:GH-sub} is that while the trapezoidal rule achieves around $\calO(n^{-p-0.8})$, Gauss--Hermite quadrature achieves only a slower convergence rate almost $\calO(n^{-p/2-0.5})$, where $n$ is the number of quadrature points. 
A similar empirical inefficiency of Gauss--Hermite quadrature is reported in a paper by one of the present authors and Nuyens~\cite{NS2021},  also in a recent paper by Trefethen~\cite{T2021}
argued that Gauss--Hermite quadrature converges more slowly than the truncated trapezoidal
rule as $n\to\infty$, because its quadrature points are unnecessarily spread out, and in effect not enough quadrature points are utilised.

In this paper, we 
\emph{prove} a sub-optimality of Gauss--Hermite quadrature. 
More precisely, we establish a sharp lower bound for the error decay of Gauss--Hermite quadrature in the sense of worst case error. Moreover, we show that a suitably truncated trapezoidal rule achieves the optimal rate of convergence, up to a logarithmic factor.

The integrands of our interest are functions with finite smoothness. 
In this regard, we will work under the assumption that the function $f$ lives in the weighted $L^2$-Sobolev space with smoothness $\alpha\in\mathbb{N}$.  
This function space is widely used; see for example the books \cite{B1998,Canuto.C_etal_2006_book,Shen.J_2011_book} and references therein. 
For $\alpha\in\mathbb{N}$, this space is equivalent to the so-called Hermite space of finite smoothness, which has been attracting attention in high-dimensional computations; see \cite{IG2015_Hermite,DILP2018,gnewuch2021countable} for its use in high-dimensional computations, and see \cite[1.5.4. Proposition]{B1998} together with Section~\ref{sec:space} below for the equivalence.

In this setting, we prove that the worst case integration error of Gauss--Hermite quadrature is bounded from below by  $n^{-\alpha/2}$ up to a constant. This rate matches the upper bound shown by Mastroianni and Monegato \cite{MM1994}, and thus cannot be improved. 
Moreover, this rate provides a rigorous verification of the numerical findings  \cite[Section 4]{DILP2018}, where the authors computed  approximate values of the worst-case error for Gauss--Hermite quadrature in the Hermite space of finite smoothness, and observed the rate $\calO(n^{-\alpha/2})$ for  $\alpha=1,2,$ and $3$.

In the proof, we exploit the structure of the Gauss--Hermit quadrature \textit{points}. The argument is independent of the quadrature weights, and thus tuning them does not change the result. 
The proof in particular indicates that, if the spacing of a node set decreases asymptotically no faster than $1/\sqrt{n}$, then the corresponding quadrature rule cannot achieve the worst-case error better than $O(n^{-\alpha/2})$; see the proof of Lemma~\ref{lem:lower_bound_general} together with  Theorem~\ref{thm:GH-LB}.

It turns out that this rate is merely half of the best possible: 
if we allow $n$ quadrature-points and weights to be arbitrary, then the best achievable using (linear) quadrature is of the rate $\calO(n^{-\alpha})$; see \cite[Theorem 1]{DILP2018} for a precise statement.
Dick et al.~\cite{DILP2018} also show that a class of numerical integration methods based on so-called (scaled) higher order digital nets achieve the optimal rate $\calO(n^{-\alpha})$ up to a logarithmic factor in the multi-dimensional setting, including the one-dimensional setting as a special case.

Our results on the trapezoidal rule show that the trapezoidal rule, a method arguably significantly simpler than one-dimensional higher order digital nets,
also achieves this optimal rate, up to a logarithmic factor, and thus is nearly \textit{twice} as fast as the error-decay rate of Gauss--Hermite rule. 
It is also worth mentioning that Gauss--Hermite quadrature requires a nontrivial algorithm to generate quadrature points, whereas for the trapezoidal rule we simply have equispaced points.

For analytic functions, the efficiency of the trapezoidal rule is well known;  
related studies date back at least to the paper by Goodwin in 1949 \cite{Goodwin1949}, and this accuracy is not only widely known, but also still actively studied by contemporary numerical analysts \cite{S1997,MS2001,GilEtAl2007,Waldvogel2011,TW2014,T2021}. 
Our results show that this efficiency extends to Sobolev class functions, where we do not have tools from complex analysis such as contour integrals. 
Our proof uses the strategy recently developed by one of the present authors and Nuyens~\cite{NS2021} for a class of quasi-Monte Carlo methods.

We now mention other error estimates for Gauss--Hermite quadrature in the literature.  
Based on results by Freud~\cite{Freud1972}, Smith, Sloan, and Opie~\cite{SSO1983} showed an upper bound $\calO(n^{-\alpha/2})$ for $\alpha$-times continuously differentiable functions whose $\alpha$-th derivative satisfies a suitable growth condition for $\alpha\in\mathbb{N}$. 
Since the weighted Sobolev space seems to be more frequently used, our focus is on this class. 
Della Vecchia and Mastroianni~\cite{DM2003} showed an upper bound $\calO(n^{-1/6})$ for $\rho$-integrable a.e.~differentiable functions whose derivative is also $\rho$-integrable. 
Moreover, their result implies a matching lower bound in the sense of worst-case error, and thus  in this sense their upper bound is sharp. 
It does not seem to be trivial to determine whether their bounds generalise, for example to the order $n^{-\alpha/6}$ or to $n^{-\alpha/2+1/3}$, with the $\alpha$-times differentiability of the integrand for $\alpha\ge2$.
In contrast, our results show that, by assuming the square $\rho$-integrability, the rate improves to $\calO(n^{-\alpha/2})$, for general $\alpha\in\mathbb{N}$.

For analytic functions, 
the rate $\exp(-C\sqrt{n})$ with a constant $C>0$ has been mentioned  in the literature for functions with a \textit{suitable decay} 
that are analytic in a strip region.   
Barrett~\cite{Barrett.W_1961} seems to be the first to have presented this rate; see also Davis and Rabinowitz \cite[Equation~(4.6.1.18)]{DR1984}. 
Note, however, that no explicit proof or statement is given in either of these references; in particular, the decay condition for which this rate holds is not given. 
We are not aware of any reference that gives an explicit statement with complete assumptions. 
On the other hand, for the trapezoidal rule, Sugihara~\cite{S1997} conducted an extensive research on the integration error for functions analytic in a strip region with various decay conditions. 
In particular, for functions decaying at the rate $\exp(-\tilde{C}|x|^\rho)$ ($\rho\geq 1$) on the real axis, under other suitable assumptions he established the rate $\exp(-{C}^*n^{\rho/(\rho+1)})$, with an explicit constant ${C}^*>0$; see \cite[Theorem 3.1]{S1997} for a precise statement.  
Hence, whatever the decay condition the Gauss--Hermite rule requires to give the rate $\exp(-C\sqrt{n})$ may be, for $\rho\geq 1$ the trapezoidal rule attains a faster rate for the function class considered in \cite{S1997}. 
Note that lower bounds for the integration error are also presented in \cite{S1997}, which shows that the trapezoidal rule is near-optimal in the function class considered there. 
Trefethen \cite{T2021} makes the comparison of the trapezoidal rule and Gauss--Hermite quadrature explicit. 
In \cite[Theorem 5.1]{T2021}, for the trapezoidal rule and various other quadrature rules he established the rate 
$\exp(-C{n}^{2/3})$ for a class of functions that are analytic in a strip region which decays at the rate $\exp(-x^2)$ on the real axis. 
He also presents a numerical result for integrating $\cos(x^3)\exp(-x^2)$ with the physicists' Gauss--Hermite rule, which supports the rate $\exp(-C\sqrt{n})$ for the Gauss--Hermite rule.

Before moving on, we mention that our results motivate further studies of Gauss--Hermite based algorithms in high-dimensional problems. 
Integration problems in high dimensions arise, for example in computing statistics of solutions of partial differential equations parametrised by random variables. 
In particular, integration with respect to the Gaussian measure has been attracting increasing attention; see for example \cite{Graham.I_eta_2015_Numerische,C2018,YK2019,HS2019ML,D2021,dung2022analyticity}. 
The measure being Gaussian, algorithms that use Gauss--Hermite quadrature as a building block have gained popularity \cite{C2018,D2021,dung2022analyticity}.  
The key to proving error estimates in this context is the regularity of the quantity of interest with respect to the parameter, and for elliptic and parabolic problems such smoothness, even an analytic regularity, has been shown  \cite{BNT2010SIREV,NT2009parabolic,dung2022analyticity}. 
In contrast, the solutions of parametric hyperbolic systems suffer from limited regularity under mild assumptions \cite{MNT2013_StochasticCollocationMethod,MNT2015_AnalysisComputationElasticWave}. 
Hence, our results, which show the sub-optimality of Gauss--Hermite rule for functions with finite smoothness, caution us and encourage further studies of the use of algorithms based on Gauss--Hermite rule for this class of problems.

Finally, we note that if we have the weight function 
$\mathrm{e}^{-x^2}$ instead of $\mathrm{e}^{-x^2/2}$,
the corresponding orthogonal polynomials are called physicist's Hermite polynomials. Our results for Gauss--Hermite quadrature can be obtained for these polynomials by simply rescaling our results by $x\mapsto \sqrt{2}x.$
Likewise, results for physicist's Hermite polynomials in the literature, e.g.~in \cite{Szego1975}, are used throughout this paper. 

The rest of this paper is organized as follows. In Section~\ref{sec:space} we introduce necessary definitions such as Hermite polynomials, the weighted Sobolev spaces, and the Hermite spaces. 
We also discuss the norm equivalence between the weighted Sobolev space and the Hermite space. 
In Section~\ref{sec:GH-bounds} the sub-optimality of Gauss--Hermite quadrature is shown. In particular, we obtain matching lower and upper bounds for the worst-case error.
In Section~\ref{sec:trapez} the optimality of the trapezoidal rule is shown.
Section~\ref{sec:conc} concludes this paper.

\section{Function spaces with finite smoothness} \label{sec:space}
Throughout this paper, we use the weighted space $L^2_\rho:=L^2_\rho(\mathbb{R})$, the normed space consisting of the equivalence classes of Lebesgue measurable functions $f\colon\mathbb{R}\to\mathbb{R}$ satisfying $\|f\|_{L^2_\rho}^2:=\int_{\mathbb{R}}|f(x)|^2\rho(x)\mathrm{d}x<\infty$, where the equivalence relation is given by $f\sim g$ if and only if $\|f-g\|_{L^2_\rho}=0$.
\subsection{Hermite polynomials}
For $k\in\mathbb{N}\cup\{0\}$, the $k$-th degree probabilist's Hermite polynomial is given by
\begin{align}\label{eq:Hermite-def}
    H_{k}(x)
	=\frac{(-1)^{k}}	{\sqrt{k!}}
    	\rme^{x^{2}/2}
        \frac{\rd^{k}}{\rd x^{k}}
        	\rme^{-x^{2}/2},\quad x\in\mathbb{R},
\end{align}
where they are normalised so that $\|H_k\|_{L^2_\rho}=1$ for all $k\in\mathbb{N}\cup\{0\}$. 
The polynomials $(H_k)_{k\geq 0}$ form a complete orthonormal system for $L^2_\rho$.

The following properties are used throughout the paper.
\begin{align}\label{eq:Hermite-deriv}
H_k'(x)=\sqrt{k}H_{k-1}(x),\; k\ge1;
\end{align}
\begin{align}\label{eq:Hermite-rho-deriv}\
\frac{\rd^{\tau}}{\rd x^{\tau}}\left(H_k(x)\rho(x)\right)
&=        	
\frac{(-1)^{k}}	{\sqrt{2\pi k!}}
    	\rme^{x^{2}/2}
        \frac{\rd^{k+\tau}}{\rd x^{k+\tau}}
        	\rme^{-x^{2}/2} \nonumber
        	\\
        	&=
        	(-1)^\tau\sqrt{\frac{(k+\tau)!}{k!}}H_{k+\tau}(x)\rho(x),\; k\ge0,\; \tau\ge0.
\end{align}
\subsection{Weighted Sobolev space}
The function space we consider is the Sobolev space of square integrable functions, the integrability condition of which is imposed by the Gaussian measure.
\begin{Definition}[Weighted Sobolev space]\label{def:Sobolev}
For $\alpha\in\NN$, the weighted Sobolev space $\Hscr_{\alpha}$(with the weight function $\rho$) is the class of all functions $f\in L^2_\rho$ such that $f$ has weak derivatives satisfying $f^{(\tau)}\in L^2_\rho$ for $\tau=1,\dots,\alpha$:
\begin{align*}
 \Hscr_{\alpha}:=
    \Biggl\{
    f \in L^2_\rho \;\bigg|\; \|f\|_\alpha := \biggl(\sum_{\tau=0}^\alpha \|f^{(\tau)}\|^2_{L^2_\rho}\biggr)^{1/2} < \infty
    \Biggr\}.
\end{align*}
\end{Definition}
Elements in $\Hscr_{\alpha}$ for $\alpha\in \NN$ are in the standard local Sobolev space $W^{1,2}_{\mathrm{loc}}(\mathbb{R})$, and thus admit a continuous representative. 
In what follows, we always take the continuous representative of $f\in \Hscr_{\alpha}$.

We recall another important class of functions, the so-called Hermite space.
For this space we follow the definition of \cite{DILP2018}.
\begin{Definition}[Hermite space with finite smoothness]\label{def:Hermite}
For $\alpha\in\NN$, the Hermite space with finite smoothness $\cH^{\mathrm{Hermite}}_{\alpha}$ is given by
\begin{align*}
 \cH^{\mathrm{Hermite}}_{\alpha}:=
    \Bigg\{
    f \in L^2_\rho \;\bigg|\; \|f\|_{\cH^{\mathrm{Hermite}}_{\alpha}} := 
        \biggl(
            \sum_{k=0}^\infty r_{\alpha}(k)^{-1}\bigl|\widehat{f}(k)\bigr|^2
        \biggr)^{1/2} < \infty
    \Bigg\},
\end{align*}
where $\widehat{f}(k)=(f,H_k)_{L^2_{\rho}}:=\int_\R f(x)H_k(x) \rho(x)\rd x$, and
\[
r_\alpha(k):=\begin{cases}1, &\mathrm{if }\; k=0,\\ \bigl(\sum^\alpha_{\tau=0}\beta_\tau(k)\bigr)^{-1}, &\mathrm{if }\;k\ge1, \end{cases} 
\;\mathrm{and}\;
\beta_\tau(k):=\begin{cases} \frac{k!}{(k-\tau)!}, &\mathrm{if }\;k\ge\tau,\\ 0, &\mathrm{otherwise.} \end{cases}
\]
\end{Definition}
It turns out $\mathscr{H}_{\alpha}=\cH^{\mathrm{Hermite}}_{\alpha}$, where the equality here means the norm equivalence. 
Hence, results established for the Hermite space $\cH^{\mathrm{Hermite}}_{\alpha}$ can be readily translated to $\mathscr{H}_{\alpha}$ up to a constant, which allows us to compare the results on higher order digital nets in \cite{DILP2018} with ours.

A proof of this equivalence of the norm is outlined in \cite[1.5.4.~Proposition]{B1998}. 
A more detailed proof of one direction of the equivalence,  $f\in\cH^{\mathrm{Hermite}}_{\alpha}$ implying $f\in\Hscr_{\alpha}$ with the same $\alpha\in\NN$, is given in \cite[Lemma~6]{DILP2018}. 
Here, for completeness we prove its converse.
\begin{lemma}
Let $f\in\Hscr_{\alpha}$ with $\alpha\in\NN$, then $f\in\cH^{\mathrm{Hermite}}_{\alpha}$ with the same smoothness parameter $\alpha$.
\end{lemma}
\begin{proof}
We first prove the claim for $\alpha=1$. Assume $f\in\Hscr_{1}$. Let 
$F(x):=f(x)\; (\rho (x))^{1/2+\varepsilon}$ and $ G(x):=H_k(x)(\rho (x))^{1/2-\varepsilon}$ for $0<\varepsilon<1/2$.
We have
\[\int_{\R}F'(x) \phi(x) \rd x
=
-  \int_{\R}F(x) \phi'(x) \rd x,
\]
for any function $\phi$ in the space of compactly supported infinitely differentiable functions $C^\infty _{\mathrm{c}} (\R)$.
Since $G$ is in the standard Sobolev space $W^{1,2} (\R)$, 
there exists a sequence $\{\phi_N\}_{N\in\mathbb{N}}\subset C^\infty _{\mathrm{c}} (\R)$ that satisfies
\[
 \|G-\phi_N\|^2_{L^2(\R)} +  \|G'-\phi'_N\|^2_{L^2(\R)} = \|G-\phi_N\|^2_{W^{1,2}(\R)} \le 1/N.
\]
Then, letting $g\in L^2(\mathbb{R})$, the Cauchy--Schwarz inequality implies 
$\int_{\R} |F(x)g(x)| \rd x
     \le
     \|f\|_{L^2_\rho}\|g\|_{L^2(\R)}<\infty$, 
and for
$F'(x)=f'(x) (\rho (x))^{1/2+\varepsilon}-(\varepsilon+1/2)xf(x)(\rho (x))^{1/2+\varepsilon}$ we also have
\begin{align*}
     \int_{\R} |F'(x)g(x)| \rd x
     &\le
     \left(\int_{\R} | f'(x) |^2 \rme^{-x^2/2} \rd x \right)^{1/2}\;
     \left(\int_{\R}  |g(x)|  \rd x \right)^{1/2}
\\
     &\kern-1.5cm +
     \sup_{t\in\mathbb{R}}|(\varepsilon+1/2)^2t^2 \rme^{-\varepsilon t^2}|
     \left(\int_{\R} | f(x) |^2 \rme^{-x^2/2} \rd x \right)^{1/2}\;
     \left(\int_{\R}  |g(x)|^2 \rd x \right)^{1/2}
<\infty.
\end{align*}
Therefore both $\langle F,\cdot\rangle:=\int_{\R}F(x) \cdot \rd x$ and $\langle F',\cdot \rangle:=\int_{\R}F'(x) \cdot \rd x$ define a continuous functional on $L^2(\R)$. 
Hence, we have $\int_{\R}F'(x)G(x)\rd x=-\int_{\R}F(x)G'(x)\rd x$ and thus
\begin{align*}
 &\int_{\R} f'(x) H_k(x) \rho(x) -(\varepsilon+1/2)x H_k (x)f(x)\rho (x) \rd x
 =\int_{\R} F'(x)G(x) \rd x\\
 =&-\int_{\R} F(x)G'(x) \rd x=-
 \biggl(\!\int_{\R}\!f(x) \sqrt{k}H_{k-1}(x) \rho(x) -(-\varepsilon+1/2)x H_k (x)f(x)\rho (x) \rd x
 \biggr),
\end{align*}
which is equivalent to 
\begin{align}
  \int_{\R} f'(x)& H_k(x) \rho(x)=-\int_{\R} f(x) \sqrt{k}H_{k-1}(x) \rho(x)-xH_k (x)f(x)\rho (x) \rd x
  \label{eq:<f',H>-identity}
  \\
  &=
  -\int_{\R} f(x) (H_k (x)\rho (x))' \rd x=
  \int_{\R} f(x) \sqrt{k+1}H_{k+1} (x)\rho (x) \rd x
  ,\nonumber
\end{align}
where we used $(H_k (x)\rho (x))'=-\sqrt{k+1}H_{k+1} (x)\rho (x)$. Hence we obtain
\begin{align*}
 (f',H_k)_{L^2_\rho} ^2 =(k+1)\;(f,H_{k+1})_{L^2_\rho} ^2,
\end{align*}
and
\[
 \sum_{k=0} ^\infty (k+1)\;(f,H_{k+1})_{L^2_\rho} ^2=\sum_{k=0} ^\infty (f',H_k)_{L^2_\rho}^2=\|f'\|_{L^2_\rho} ^2 <\infty.
\]
This implies $f\in\cH^{\mathrm{Hermite}}_{1}$ since $r_1^{-1}(k)=k+1$.
For general $\tau=1,...,\alpha$, assuming $f\in\Hscr_\tau$ and by repeating the same argument as above, 
we have $(f^{(\tau)},H_{k})_{L^2_\rho}^2=(f,H_{k+\tau})_{L^2_\rho}^2\prod_{j=1}^{\tau}(k+j)$ for $k\geq0$ and thus
\begin{align*}
\|f^{(\tau)}\|_{L^2_\rho}^2
&=
\sum_{k=0}^{\infty}
    (f,H_{k+\tau})_{L^2_\rho}^2\prod_{j=1}^{\tau}(k+j)
\geq\frac{1}{\tau^{\tau}}\sum_{k=0}^{\infty}(f,H_{k+\tau})_{L_{\rho}^{2}}^{2}(k+\tau)^{\tau}\\
&
=\frac{1}{\tau^{\tau}}\sum_{k=\tau}^{\infty}k^{\tau}(f,H_{k})_{L_{\rho}^{2}}^{2}
.
\end{align*}
Hence, observing $\lim_{k\to\infty }r_\tau(k)\,k^{\tau}=1$ (see also \cite[p.~687]{DILP2018}), we conclude $f\in \cH^{\mathrm{Hermite}}_{\tau}$.
\end{proof} 

\section{Matching bounds for Gauss--Hermite quadrature}\label{sec:GH-bounds}
In this section, we prove the sub-optimality of Gauss--Hermite quadrature. 
We first introduce the following linear quadrature of general form
\begin{align}\label{eq:general_quadrature}
    Q_n(f)=\sum_{j=1}^{n}w_{j} f(\xi_{j})
\end{align}
with arbitrary $n$ distinct quadrature points on the real line
\[ -\infty<\xi_1<\xi_2<\cdots <\xi_n<\infty \]
and quadrature weights $w_1,\ldots,w_n\in \RR$.
Gauss--Hermite quadrature $Q_{n}^{\GH}$ is given by the points $(\xi_{j}^{\GH})_{j=1,\dots,n}$ being the roots of $H_n$ and the weights $(w_{j})_{j=1,\dots,n}$ being  
$
w_{j}={1}/{[H_{n}'(\xi_{j}^{\GH})]^{2}}
$, 
see for example \cite[Theorem 3.5]{Shen.J_2011_book}.

Given a quadrature rule $Q_n$, 
it is convenient to introduce the notation
\[
{e}^{\wor}(Q_{n}, \Hscr_{\alpha})
:=
\sup_{0\neq f\in \Hscr_{\alpha}}
\frac{|I(f)-Q_n(f)|}
    {\|f\|_{\alpha}}
.
\]
The quantity ${e}^{\wor}(Q_{n}, \Hscr_{\alpha})$ is commonly referred to as the \emph{worst-case error} of $Q_n$ in $\Hscr_{\alpha}$; see for example \cite{Dick.J_Kuo_Sloan_2013_ActaNumerica}.
Now we can state our aim of this section more precisely: we prove the matching lower and upper bounds on ${e}^{\wor}(Q_{n}^{\GH}, \Hscr_{\alpha})$ for Gauss--Hermite quadrature.
\subsection{Lower bound}
We first derive the following lower bound on ${e}^{\wor}(Q_{n}, \Hscr_{\alpha})$ for the general quadrature \eqref{eq:general_quadrature}.
\begin{lemma}\label{lem:lower_bound_general}
Let $\alpha\in\NN$.
For $n\geq 2$, let
\begin{align}\label{eq:assum_minimum_distance}
    \sigma := \begin{cases} \alpha & \text{if }\ \min_{j=1,\ldots,n-1}(\xi_{j+1}-\xi_{j})\leq 1.\\
    0 & \text{otherwise.}\end{cases}
\end{align}
Then, there exists a constant $c_{\alpha}>0$, which depends only on $\alpha$, such that the worst-case error ${e}^{\wor}(Q_n, \Hscr_{\alpha})$ of a general function-value based linear quadrature  \eqref{eq:general_quadrature} in the weighted Sobolev space $\mathscr{H}_{\alpha}$ is bounded below by
\begin{align*}
    &{e}^{\wor}(Q_n, \Hscr_{\alpha})\geq c_{\alpha}\min_{i=1,\dots,n-1}(\xi_{i+1}-\xi_{i})^{\sigma+1/2}
    \\&
    \qquad\quad\times \sum_{j=1}^{n-1}\mathrm{e}^{-\max(\xi_{j}^2,\xi_{j+1}^2)/2}
        \left(\sum_{k=1}^{n-1}\mathrm{e}^{-\dsone_{\geq 0}(\xi_{k}\xi_{k+1}) \min(\xi_{k}^2,\xi_{k+1}^2)/2}\right)^{-1/2},
\end{align*}
where $\dsone_{\geq 0}(x)$ is equal to 1 if $x\geq 0$ and 0 otherwise.
\end{lemma}
\begin{proof}
The heart of the matter is to construct a function $0\neq h_{n}\in \Hscr_{\alpha}$ such that $h_{n}(\xi_{j})=0$ for all $1\leq j\leq n$, resulting in $Q_n(h_{n})=0$, and that  $\|h_n\|_{\alpha}$ is small but $I(h_n)$ is large.
Define a function $h\colon \RR\to \RR$ by

\[ h(x):=h_n(x)
:=\begin{cases} 
    \displaystyle
    \biggl(\frac{x-\xi_{j}}{\xi_{j+1}-\xi_{j}}\biggr)^{\!\!\alpha}  \biggl(1-\frac{x-\xi_{j}}{\xi_{j+1}-\xi_{j}}\biggr)^{\!\!\alpha} 
        &\kern-4mm  
        \begin{array}{l}
        \text{if there exists }j\!\in\! \{1,\ldots,n-1\}\\[-2.5pt]
        \text{such that }x\in [\xi_{j},\xi_{j+1}],
        \end{array}
        \\[10pt]
    0  &\kern-4mm\begin{array}{l}\text{otherwise.}\end{array}
\end{cases}
\]
Then, $h$ turns out to fulfill our purpose. This type of fooling function used to prove lower bounds for the worst-case error is called \emph{bump function} (of finite smoothness), in quasi-Monte Carlo theory; see, for instance, \cite[Section~2.7]{DHP2015}.

First we show $h\in \Hscr_{\alpha}$. It follows from
\[ \left(\frac{x-\xi_{j}}{\xi_{j+1}-\xi_{j}}\right)^{\alpha}\left(1-\frac{x-\xi_{j}}{\xi_{j+1}-\xi_{j}}\right)^{\alpha}=\sum_{\ell=0}^{\alpha}(-1)^{\ell}\binom{\alpha}{\ell}\left(\frac{x-\xi_{j}}{\xi_{j+1}-\xi_{j}}\right)^{\alpha+\ell}\]
that we have
\begin{align*}
    h^{(\tau)}(x)=\frac{1}{(\xi_{j+1}-\xi_{j})^{\tau}}\sum_{\ell=0}^{\alpha}(-1)^{\ell}\binom{\alpha}{\ell}\frac{(\alpha+\ell)!}{(\alpha+\ell-\tau)!}\left(\frac{x-\xi_{j}}{\xi_{j+1}-\xi_{j}}\right)^{\alpha+\ell-\tau}
\end{align*}
for $\tau=0,1,\ldots,\alpha$ and any $\xi_{j}<x<\xi_{j+1}$. 
As we have $h^{(\tau)}(\xi_{j})=h^{(\tau)}(\xi_{j+1})=0$ for $\tau=0,1,\ldots,\alpha-1$, the function $h$ is $(\alpha-1)$-times continuously differentiable. Moreover, $h^{(\alpha-1)}$ is continuous piecewise polynomial and thus weakly differentiable. Also, noting that $\dsone_{\geq 0}(\xi_{j}\xi_{j+1})=0$ only when $\xi_{j}<0<\xi_{j+1}$ and otherwise $\dsone_{\geq 0}(\xi_{j}\xi_{j+1})=1$, we have
\begin{align*}
    & \int_{\xi_{j}}^{\xi_{j+1}}| h^{(\tau)}(x)|^2\rho(x)\rd x \\
    & \leq \frac{\mathrm{e}^{-\dsone_{\geq 0}(\xi_{j}\xi_{j+1}) \min(\xi_{j}^2,\xi_{j+1}^2)/2}}{\sqrt{2\pi}}\int_{\xi_{j}}^{\xi_{j+1}}| h^{(\tau)}(x)|^2\rd x 
        \\
    & = \frac{\mathrm{e}^{-\dsone_{\geq 0}(\xi_{j}\xi_{j+1}) \min(\xi_{j}^2,\xi_{j+1}^2)/2}}{\sqrt{2\pi}(\xi_{j+1}-\xi_{j})^{2\tau}} \\
        &{
            \times \kern-3mm\sum_{\ell_1,\ell_2=0}^{\alpha}\kern-1.5mm(-1)^{\ell_1+\ell_2} \kern-0.5mm
            \binom{\alpha}{\ell_1}\binom{\alpha}{\ell_2}\frac{(\alpha+\ell_1)!}{(\alpha+\ell_1-\tau)!}
                \frac{(\alpha+\ell_2)!}{(\alpha+\ell_2-\tau)!}
                \int_{\xi_{j}}^{\xi_{j+1}}\kern-1mm\Bigl(\frac{x-\xi_{j}}{\xi_{j+1}-\xi_{j}}\Bigr)^{2(\alpha-\tau)+\ell_1+\ell_2}\kern-0.6mm\rd x} \\
    & = \frac{\mathrm{e}^{-\dsone_{\geq 0}(\xi_{j}\xi_{j+1}) \min(\xi_{j}^2,\xi_{j+1}^2)/2}}{\sqrt{2\pi}(\xi_{j+1}-\xi_{j})^{2\tau-1}}\\&
    \quad{\times \sum_{\ell_1,\ell_2=0}^{\alpha}
        \frac{(-1)^{\ell_1+\ell_2}}{2(\alpha-\tau)+\ell_1+\ell_2+1}\binom{\alpha}{\ell_1}\binom{\alpha}{\ell_2}\frac{(\alpha+\ell_1)!}{(\alpha+\ell_1-\tau)!}\frac{(\alpha+\ell_2)!}{(\alpha+\ell_2-\tau)!},}
\end{align*}
for $\tau=0,1,\ldots,\alpha$. The last sum over $\ell_1$ and $\ell_2$ does not depend on $j$. 
Denoting this sum by $S_{\alpha,\tau}$, we obtain
\begin{align*}
    \|h\|_{\alpha}^2 & = \sum_{\tau=0}^{\alpha}\int_{\RR}| h^{(\tau)}(x)|^2 \rho(x)\rd x = \sum_{\tau=0}^{\alpha}\sum_{j=1}^{n-1}\int_{\xi_{j}}^{\xi_{j+1}}| h^{(\tau)}(x)|^2\rho(x)\rd x \\
    & \leq \frac{1}{\sqrt{2\pi}}\sum_{\tau=0}^{\alpha}S_{\alpha,\tau}\sum_{j=1}^{n-1}\frac{\mathrm{e}^{-\dsone_{\geq 0}(\xi_{j}\xi_{j+1}) \min(\xi_{j}^2,\xi_{j+1}^2)/2}}{(\xi_{j+1}-\xi_{j})^{2\tau-1}}\\
    & \leq
        \frac{1}{\sqrt{2\pi}}
        \sum_{\tau=0}^{\alpha}
            \frac{S_{\alpha,\tau}}
            {\min_{1\le i\le n-1}(\xi_{i+1}-\xi_{i})^{2\tau-1}}\sum_{j=1}^{n-1}\mathrm{e}^{-\dsone_{\geq 0}(\xi_{j}\xi_{j+1}) \min(\xi_{j}^2,\xi_{j+1}^2)/2}\\
    & \leq
    \frac{1}{\sqrt{2\pi}\min_{1\le i\le n-1}(\xi_{i+1}-\xi_{i})^{2\sigma-1}}\sum_{\tau=0}^{\alpha}S_{\alpha,\tau}\sum_{j=1}^{n-1}\mathrm{e}^{-\dsone_{\geq 0}(\xi_{j}\xi_{j+1}) \min(\xi_{j}^2,\xi_{j+1}^2)/2}<\infty.
\end{align*}
This proves $h\in \Hscr_{\alpha}$.

By definition of $h$, we have $h(\xi_{j})=0$ for all $j=1,\ldots,n$, and thus
\[ Q_n(h)=0.\]
Moreover, we have
\begin{align*}
    I(h) & = \int_{\RR}h(x)\rho(x)\rd x = \sum_{j=1}^{n-1}\int_{\xi_{j}}^{\xi_{j+1}}h(x)\rho(x)\rd x \\
    & \geq \frac{1}{\sqrt{2\pi}}\sum_{j=1}^{n-1}\mathrm{e}^{-\max(\xi_{j}^2,\xi_{j+1}^2)/2}\int_{\xi_{j}}^{\xi_{j+1}}\left(\frac{x-\xi_{j}}{\xi_{j+1}-\xi_{j}}\right)^{\alpha}\left(1-\frac{x-\xi_{j}}{\xi_{j+1}-\xi_{j}}\right)^{\alpha}\rd x \\
    & = \frac{1}{\sqrt{2\pi}}\sum_{j=1}^{n-1}\mathrm{e}^{-\max(\xi_{j}^2,\xi_{j+1}^2)/2}(\xi_{j+1}-\xi_{j})\int_{0}^{1}x^{\alpha}\left(1-x\right)^{\alpha}\rd x \\
    & = \frac{(\alpha!)^2}{(2\alpha+1)!\sqrt{2\pi}}\sum_{j=1}^{n-1}\mathrm{e}^{-\max(\xi_{j}^2,\xi_{j+1}^2)/2}(\xi_{j+1}-\xi_{j}) \\
    & \geq \frac{(\alpha!)^2}{(2\alpha+1)!\sqrt{2\pi}}
    \min_{1\le i\le n-1}(\xi_{i+1}-\xi_{i})\sum_{j=1}^{n-1}\mathrm{e}^{-\max(\xi_{j}^2,\xi_{j+1}^2)/2}.
\end{align*}

Using the above results, we obtain
\begin{align*}
    {e}^{\wor}(Q_n,\Hscr_{\alpha}) & 
    \geq \frac{|I(h)-Q_n(h)|}{\|h\|_{\alpha}}\\
    & \geq \frac{(\alpha!)^2}{(2\alpha+1)!(2\pi)^{1/4}}\left(\sum_{\tau=0}^{\alpha}S_{\alpha,\tau}\right)^{-1/2} \min_{1\le i\le n-1}(\xi_{i+1}-\xi_{i})^{\sigma+1/2}\\
    &{\quad \times \sum_{j=1}^{n-1}\mathrm{e}^{-\max(\xi_{j}^2,\xi_{j+1}^2)/2}
    \left(\sum_{k=1}^{n-1}\mathrm{e}^{-\dsone_{\geq 0}(\xi_{k}\xi_{k+1}) \min(\xi_{k}^2,\xi_{k+1}^2)/2}\right)^{-1/2}.}
\end{align*}
Now the proof is complete.
\end{proof}
Using the general lower bound in Lemma~\ref{lem:lower_bound_general}, we obtain the following lower bound on the worst-case error for Gauss--Hermite quadrature.
\begin{theorem}\label{thm:GH-LB}
Let $\alpha\in\NN$.
For any $n\geq 2$, the worst-case error of the Gauss--Hermite quadrature in the weighted Sobolev space $\mathscr{H}_{\alpha}$ is bounded from below as
\[ {e}^{\wor}(Q_{n}^{\GH}, \Hscr_{\alpha})\geq C_{\alpha}n^{-\alpha/2}\]
with a constant $C_{\alpha}>0$ that depends on $\alpha$ but independent of $n$.
\end{theorem}
\begin{proof}
Let $\xi_{j}^{\GH}$, $j=1,\dots,n$, be the roots of $H_n$. 
For any $n\geq 2$, it holds that
\begin{align}\label{eq:minimum_distance}
\frac{\pi}{\sqrt{n+1/2}}<\min_{j=1,\dots,n-1} (\xi_{j+1}^{\GH}-\xi_{j}^{\GH})\leq \frac{\sqrt{21/2}}{\sqrt{n+1/2}};
\end{align} 
see, for instance, \cite[Eq.~(6.31.22)]{Szego1975}. 
Thus, to invoke  Lemma~\ref{lem:lower_bound_general} we let
\[ \sigma := \begin{cases} \alpha & \text{for }\ n\geq 10,\\
    0 & \text{for }\ 2\leq n<10,\end{cases}
\]
so that the conditions in \eqref{eq:assum_minimum_distance} are satisfied.
Also, each node $\xi_{j}^{\GH}$ is bounded below and above as follows; see, for instance, \cite[Eq.~(6.31.19)]{Szego1975}:
for $n$ odd, we have $\xi_{(n+1)/2}^{\GH}=0$ with the positive zeros satisfying
\begin{align}\label{eq:node_distribution_odd}
    \frac{j\pi}{\sqrt{n+1/2}} <\xi_{(n+1)/2+j}^{\GH}<\frac{4j+3}{\sqrt{n+1/2}}
    \quad\text{for }\ j=1,\ldots,(n-1)/2,
\end{align}
and for $n$ even,
\begin{align}\label{eq:node_distribution_even}
 \frac{(j-1/2)\pi}{\sqrt{n+1/2}} <\xi_{n/2+j}^{\GH}<\frac{4j+1}{\sqrt{n+1/2}}
    \quad\text{for }\ j=1,\ldots,n/2,
\end{align}
with symmetricity $\xi_{j}^{\GH} = -\xi_{n+1-j}^{\GH}$, $1\leq j\leq n$ for $n\geq 2$ odd and even. 

Let $n$ be odd. Using the result in Lemma~\ref{lem:lower_bound_general}, equations~\eqref{eq:minimum_distance}, and \eqref{eq:node_distribution_odd}, together with the symmetricity of the Hermite zeros, we obtain
\begin{align*}
    & {e}^{\wor}(Q_{n}^{\GH}, \Hscr_{\alpha}) \\
    & \geq 2c_{\alpha}
        \left( \frac{\pi^2}{n+1/2}\right)^{\sigma/2+1/4}\,\sum_{j=1}^{(n-1)/2}\mathrm{e}^{-(\xi_{(n+1)/2+j}^{\GH})^2/2}
    \left(2\sum_{k=1}^{(n-1)/2}\mathrm{e}^{-(\xi_{(n+1)/2+k-1}^{\GH})^2/2}\right)^{-1/2}\\
& \geq 2c_{\alpha}
        \left( \frac{\pi^2}{n+1/2}\right)^{\sigma/2+1/4}\,\sum_{j=1}^{(n-1)/2}\kern-2mm\mathrm{e}^{-(4j+3)^2/(2n+1)}
        \left(2+2\kern-2.5mm\sum_{k=2}^{(n-1)/2}\kern-2mm\mathrm{e}^{-\pi^2(k-1)^2/(2n+1)}\right)^{-1/2}.
\end{align*}
The sum over $j$ is further bounded below by
\begin{align*}
    \sum_{j=1}^{(n-1)/2}\mathrm{e}^{-(4j+3)^2/(2n+1)} & \geq \int_1^{(n-1)/2+1}\mathrm{e}^{-(4x+3)^2/(2n+1)}\rd x\\
    & = \frac{\sqrt{n+1/2}}{4}\int_{7/\sqrt{n+1/2}}^{(2n+5)/\sqrt{n+1/2}}\mathrm{e}^{-x^2/2}\rd x\\
    & \geq \frac{\sqrt{n+1/2}}{4}\int_{\sqrt{14}}^{11\sqrt{2}/\sqrt{7}}\mathrm{e}^{-x^2/2}\rd x\\
    & = \frac{\sqrt{\pi(n+1/2)}}{4\sqrt{2}}\left( \erf(11/\sqrt{7})-\erf(\sqrt{7})\right),
\end{align*}
where $\erf$ denotes the error function, and the last inequality holds for any odd $n\geq 3$. 
The sum over $k$ is further bounded above by
\begin{align*}
    \sum_{k=2}^{(n-1)/2}\mathrm{e}^{-\pi^2(k-1)^2/(2n+1)} & \leq \int_1^{(n-1)/2}\mathrm{e}^{-\pi^2(x-1)^2/(2n+1)}\rd x \\
    & \leq \sqrt{n+1/2} \int_{0}^{\infty}\mathrm{e}^{-\pi^2x^2/2}\rd x = \sqrt{\frac{n+1/2}{2\pi}}.
\end{align*}
Using these bounds, we have
\begin{align*}
    {e}^{\wor}(Q_{n}^{\GH}, \Hscr_{\alpha}) 
    & \geq 2c_{\alpha}
        \left( \frac{\pi^2}{n+1/2}\right)^{\sigma/2+1/4}\frac{\sqrt{\pi(n+1/2)}}{4\sqrt{2}}
        \left( \erf(11/\sqrt{7})-\erf(\sqrt{7})\right)\\
    & \quad \times \left(2+2\sqrt{\frac{n+1/2}{2\pi}}\right)^{-1/2}\\
    & \geq c_{\alpha}\pi^{\sigma+1/4} \frac{\erf(11/\sqrt{7})-\erf(\sqrt{7})}{2\sqrt{2}\sqrt{2+\sqrt{2}}} \frac{1}{(n+1/2)^{\sigma/2}}\\
    & \geq c_{\alpha}\pi^{1/4} \frac{\erf(11/\sqrt{7})-\erf(\sqrt{7})}{2^{(\alpha+5)/2}} \frac{1}{n^{\alpha/2}}.
\end{align*}

Let $n$ be even. 
As in the odd case, but now using \eqref{eq:node_distribution_even} instead of \eqref{eq:node_distribution_odd}, we obtain
\begin{align*}
    & {e}^{\wor}(Q_{n}^{\GH}, \Hscr_{\alpha}) \\
    & \geq c_{\alpha}
        \left( \frac{\pi^2}{n+1/2}\right)^{\sigma/2+1/4}
        \\
        & \qquad\times
        \Biggl(\mathrm{e}^{-(\xi_{n/2+1}^{\GH}})^2/2+2\sum_{j=2}^{n/2}\mathrm{e}^{-(\xi_{n/2+j}^{\GH})^2/2}\Biggr)
        \Biggl(1+2\sum_{k=2}^{n/2}\mathrm{e}^{-(\xi_{n/2+k-1}^{\GH})^2/2}\Biggr)^{-1/2}\\
    & \geq c_{\alpha}\left( \frac{\pi^2}{n+1/2}\right)^{\sigma/2+1/4}\\
    & \qquad\times
    \Biggl(\mathrm{e}^{-5^2/(2n+1)}+2\sum_{j=2}^{n/2}\mathrm{e}^{-(4j+1)^2/(2n+1)}\Biggr)
    \Biggl(1+2\sum_{k=2}^{n/2}\mathrm{e}^{-\pi^2(k-3/2)^2/(2n+1)}\Biggr)^{-1/2}.
\end{align*}
The sum over $j$ is equal to 0 for $n=2$ and is bounded below by
\begin{align*}
    \sum_{j=2}^{n/2}&\mathrm{e}^{-(4j+1)^2/(2n+1)}  \geq \int_2^{n/2+1}\mathrm{e}^{-(4x+1)^2/(2n+1)}\rd x\\
    & = \frac{\sqrt{n+1/2}}{4}\int_{9/\sqrt{n+1/2}}^{(2n+5)/\sqrt{n+1/2}}\mathrm{e}^{-x^2/2}\rd x\\
    & \geq \frac{\sqrt{n+1/2}}{4}\int_{3\sqrt{2}}^{13\sqrt{2}/3}\mathrm{e}^{-x^2/2}\rd x = \frac{\sqrt{\pi(n+1/2)}}{4\sqrt{2}}\left( \erf(13/3)-\erf(3)\right),
\end{align*}
for $n\geq 4$. 
Noting that we have $\mathrm{e}^{-5}\geq \sqrt{5\pi}\left( \erf(13/3)-\erf(3)\right)/4$, it holds that
\[ \mathrm{e}^{-5^2/(2n+1)}+2\sum_{j=2}^{n/2}\mathrm{e}^{-(4j+1)^2/(2n+1)} \geq \frac{\sqrt{\pi(n+1/2)}}{2\sqrt{2}}\left( \erf(13/3)-\erf(3)\right),\]
for any even $n$. 
The sum over $k$ is again equal to 0 for $n=2$ and is bounded above by
\begin{align*}
    \sum_{k=2}^{n/2}\mathrm{e}^{-\pi^2(k-3/2)^2/(2n+1)} & = \sum_{k=1}^{n/2-1}\mathrm{e}^{-\pi^2(k-1/2)^2/(2n+1)}  \leq \int_0^{ n/2-1}\mathrm{e}^{-\pi^2(x-1/2)^2/(2n+1)}\rd x \\
    & \leq \sqrt{n+1/2} \int_{-\infty}^{\infty}\mathrm{e}^{-\pi^2x^2/2}\rd x = \sqrt{\frac{2n+1}{\pi}}.
\end{align*}
It follows from these bounds on the sums that
\begin{align*}
    {e}^{\wor}(Q_{n}^{\GH}, \Hscr_{\alpha}) & \geq c_{\alpha}\left( \frac{\pi^2}{n+1/2}\right)^{\sigma/2+1/4}\frac{\sqrt{\pi(n+1/2)}}{2\sqrt{2}}\\
    &\qquad\times
    \bigl( \erf(13/3)-\erf(3)\bigr)
        \left(1+2\sqrt{\frac{2n+1}{\pi}}\right)^{-1/2} \\
    & \geq c_{\alpha}\pi^{\sigma+1/4} \frac{\erf(13/3)-\erf(3)}{2\sqrt{2}\sqrt{2\sqrt{2}+2}}\frac{1}{(n+1/2)^{\sigma/2}}\\
    & \geq c_{\alpha}\pi^{1/4} \frac{\erf(13/3)-\erf(3)}{2^{(\alpha+6)/2}}\frac{1}{n^{\alpha/2}}.
\end{align*}

Altogether, we obtain a lower bound for the worst-case error 
\[ {e}^{\wor}(Q_{n}^{\GH}, \Hscr_{\alpha})\geq C_{\alpha}n^{-\alpha/2}\]
with
\begin{align*}
    C_{\alpha}
    &=c_{\alpha}\pi^{1/4} \min\Biggl\{\frac{\erf(11/\sqrt{7})-\erf(\sqrt{7})}{2^{(\alpha+5)/2}}, \frac{\erf(13/3)-\erf(3)}{2^{(\alpha+6)/2}}\Biggr\}\\
    &=
    c_{\alpha}\pi^{1/4} \frac{\erf(13/3)-\erf(3)}{2^{(\alpha+6)/2}},
\end{align*}
which holds for all $n\geq 2$.
\end{proof}
The general lower bound in Lemma~\ref{lem:lower_bound_general} depends on the set of quadrature points but not on the set of weights. 
Because the lower bound for Gauss--Hermite quadrature in Theorem~\ref{thm:GH-LB} is built up on Lemma~\ref{lem:lower_bound_general}, the sub-optimality of Gauss--Hermite quadrature holds irrespective of the choice of the quadrature weights. 

The proof of Lemma~\ref{lem:lower_bound_general} in particular indicates that, if the spacing of a node set decreases asymptotically no faster than $1/\sqrt{n}$, then the corresponding quadrature rule cannot achieve the worst-case error better than $O(n^{-\alpha/2})$. 
To elaborate this point, we present the following less tight but more general result, which implies that any function-value based quadrature rule that does not have a quadrature point in $[0,n^{-1/2}]$, say, cannot have a worst-case error better than $O(n^{-\alpha/2-1/4})$.
\begin{corollary}\label{cor:delta-LB}
For $f\in\mathscr{H}_{\alpha}$ with $\alpha\in\mathbb{N}$, 
let $Q_n(f)$ be a quadrature of the form~\eqref{eq:general_quadrature}. 
Take a positive number 
$\delta=\delta(n)\in(0,1]$
such that no quadrature
point is in $(0,\delta)$. 
 Then, we have
\[
{e}^{\wor}(Q_n, \Hscr_{\alpha})\geq C_{\alpha}\delta^{\alpha+1/2},
\]
where the constant $C_{\alpha}>0$ is independent of $\delta$. In
particular, if ${Q}_{n}$ does not have any quadrature point in $(0,n^{-r})$, $r>0$, then we have ${e}^{\wor}(Q_n, \Hscr_{\alpha})\geq C_{\alpha}n^{-r\alpha-r/2}$.
\end{corollary}
\begin{proof}
Consider a function $f_{\delta,\alpha}$ defined by  $f_{\delta,\alpha}(x):=(x/\delta)^{\alpha}(1-x/\delta)^{\alpha}\mathds{1}_{[0,\delta]}(x)$,
$x\in\mathbb{R}$. Then, following the proof of Lemma~\ref{eq:general_quadrature},
analogous calculations show that we have  $|I(f_{\delta,\alpha})-Q_{n}(f_{\delta,\alpha})|=|I(f_{\delta,\alpha})|\geq C_\alpha\|f_{\delta,\alpha}\|_{\alpha}\delta^{\alpha+1/2}$ 
for a constant $C_\alpha>0$. 
This completes the proof.
\end{proof}

In passing, we note that the Theroem~\ref{thm:GH-LB} also gives a lower bound for the interpolation $L^1_{\rho}$-error. 
Given a function $f\colon\mathbb{R}\to\mathbb{R}$, let $\Lambda_{n}f$
be the polynomial interpolant defined by the zeros of $H_{n}$. 
Then,
like other interpolatory quadratures, Gauss--Hermite quadrature satisfies
$\frac{1}{\sqrt{2\pi}}\int_{\mathbb{R}}\Lambda_{n}f(x)\,\mathrm{e}^{-x^{2}/2}\rd x=\sum_{j=1}^{n}w_{j}f(x_{j})$.
Therefore, 
with 
$
 \|f-\Lambda_{n}f\|_{L_{\rho}^{1}(\mathbb{R})}
 :=\int_{\mathbb{R}}\big|f(x)-\Lambda_{n}f(x)\big|\,\rho(x)\mathrm{d}x$
denoting the interpolation $L^1_{\rho}$-error, we have
\begin{align*}
|I(f)-Q^{\mathrm{GH}}_n(f)|&=\Bigl|\int_{\mathbb{R}}\big[f(x)-\Lambda_{n}f(x)\big]\,\rho(x)\mathrm{d}x\Bigr|\leq
 \|f-\Lambda_{n}f\|_{L_{\rho}^{1}(\mathbb{R})},
\end{align*}
and thus $\sup_{0\not=f\in \mathscr{H}_{\alpha}}\frac{\|f-\Lambda_{n}f\|_{L_{\rho}^{1}(\mathbb{R})}}{\|f\|_\alpha}\geq C_\alpha n^{-\alpha/2}$.

\subsection{Upper bound}
In the previous section, we showed a lower bound for the worst-case error of the rate $n^{-\alpha/2}$. 
A matching upper bound has been shown by Mastroianni and Monegato, the result of whom we adapt to our setting. 
\begin{proposition}[\cite{MM1994}]
Let $\alpha\in\NN$ and $n\in\mathbb{N}$. For  $f\in\mathscr{H}_{\alpha}$, let $Q_{n}^{\GH}(f)$ be the Gauss--Hermite approximation to $I(f)$. Then, we have
\[
|I(f)-Q_{n}^{\GH}(f)|\leq C n^{-\alpha/2}\|f\|_{\alpha},
\]
where $C>0$ is a constant independent of $n$ and $f$.
\end{proposition}
\begin{proof}
First, $f\in\mathscr{H}_{\alpha}$ implies $f^{(\tau)}\in W^{1,2}_{\mathrm{loc}}(\mathbb{R})$ and thus $f^{(\tau)}$ admits a locally absolutely continuous representative for $\tau=0,\dots,\alpha-1$. Moreover, 
from $f\in\mathscr{H}_{\alpha}$, for any $0<\varepsilon<1/2$ we have
\[
\int_{\mathbb{R}}
	\mathrm{e}^{\varepsilon x^2/2}|f^{(\alpha)}(x)|\mathrm{e}^{-x^2/2}\mathrm{d}x
    \leq 
    \Bigl(
    	\int_{\mathbb{R}}
		|f^{(\alpha)}(x)|^2\mathrm{e}^{-x^2/2}
        \mathrm{d}x\Bigr)^{1/2}
    \Bigl(
    	\int_{\mathbb{R}}
    	\mathrm{e}^{-x^2(1/2-\varepsilon)}
     \mathrm{d}x\Bigr)^{1/2}<\infty.
\]
Thus, from \cite[Theorem 2]{MM1994}, the statement follows.
\end{proof} \section{Optimality of the trapezoidal rule} \label{sec:trapez}
In this section, we prove the optimality of trapezoidal rules in $\mathscr{H}_\alpha$.
More precisely, we consider the following quadrature with  $n$ equispaced points
\begin{align}
Q_{n,T}^{{*}}(g):=\frac{2T}{n}\sum_{j=0}^{n-1}g(\xi_j^*), \quad\text{with}\ \xi_j^*:=\frac{2T}{n}j-T,\ j=0,\dots,n-1.\label{eq:def-trap}
\end{align}
Here, $T>0$ is a parameter that controls the cut-off of the integration domain from $\mathbb{R}$ to $[-T,T]$. 

We call $Q_{n,T}^*$ a \textit{trapezoidal rule}:
indeed, for $n$ even, $Q_{n,T}^{*}(g)$ is nothing but the standard truncated trapezoidal rule for functions on $\mathbb{R}$ with mesh size $\Delta x$
\[
Q^*_{n,T}(g)=\Delta x \sum_{\ell=-n/2}^{n/2-1}g(\ell\, \Delta x),
\]
with $\Delta x=2T/n$, while for $n$ odd we have the trapezoidal rule only shifted by $-\frac12\Delta x$:
\[
Q^*_{n,T}(g)=\Delta x \sum_{\ell=-(n-1)/2}^{(n-1)/2}g\Bigl(\bigl(\ell-\frac12\bigr)\Delta x\Bigr).
\]

Our proof strategy is based on the approach by Nuyens and Suzuki \cite{NS2021}, where a multidimensional integration problem with respect to the Lebesgue measure using a quasi-Monte Carlo method called rank-$1$ lattice rule was considered. 
Following \cite{NS2021}, we consider the bound
\begin{align}
\left|
    \int_{\R} g(x) \rd x
    -
Q_{n,T}^*(g)
  \right|
  \le
  \biggl|\int_{\R} g(x) \rd x -\int_{-T}^T g(x) \rd x  \biggr| 
  +
  \biggl|\int_{-T}^T g(x) \rd x - Q_{n,T}^* (g) \biggr|.
  \label{eq:error-decomp-g}
\end{align}
We will let $g=f\rho$ with $f\in\Hscr_\alpha$ later in Theorem~\ref{thm:trape-opt}.

The first term of the right hand side in \eqref{eq:error-decomp-g} can be bounded as in \cite[Proposition~8]{NS2021}; 
we provide a proof adapted to our setting below in Proposition~\ref{prop:trape}. 
To bound the second term, we now derive what corresponds to \cite[Lemma~5 and Proposition~7]{NS2021}. 

\begin{lemma}
\label{lem:trap-error-inside}
Let $T>0$, $\alpha\in\mathbb{N}$, and $n\in\mathbb{N}$ be given.
Suppose that $g^{(\tau)}\colon\mathbb{R}\to\mathbb{R}$ is absolutely continuous on any compact interval for $\tau=0,\dots,\alpha-1$, 
and that $g^{(\alpha)}$ is in $L^2(\R)$.
Suppose further that $g$ satisfies
\begin{equation}
\|g\|_{\alpha,[-T,T]}
:=
\left( \left(\sum_{\tau=0}^{\alpha-1}  \left(\int_{-T}^T g^{(\tau)}(x) \rd x\right)^2 +\int_{-T}^T | g^{(\alpha)}(x) |^2\rd x \right)
\right)^{1/2}
< \infty
\end{equation}
and
\begin{align}
 \|g\|_{\alpha,\mathrm{decay}} 
 :=
 \sup_{\substack{x\in\R \\ \tau\in\{0,\ldots,\alpha-1\}}}
 \left| \rme^{(1-\varepsilon)x^2/2} \, g^{(\tau)}(x) \right|
 <
 \infty
 ,\quad \text{for some  $\varepsilon\in(0,1)$}.
 \label{eq:decay-cond}
\end{align}
Then the error of the $n$-point trapezoidal rule on the interval $[-T,T]$ defined in \eqref{eq:def-trap} is bounded by 
\begin{align}
\left|
    \int_{-T} ^T g(x) \rd x
    -
Q_{n,T}^*(g)
  \right| 
   &\le
   C_\alpha\|g\|_{\alpha,[-T,T]} T^{\alpha+1/2} \frac{1}{n^{\alpha}} \nonumber
   \\&\quad+
   \alpha \max\{1,(2T)^{\alpha-1}\}\|g\|_{\alpha,\mathrm{decay}}\rme^{-(1-\varepsilon)T^2/2}
  ,
\end{align}
with $C_\alpha :=2\sqrt{\zeta(2\alpha)}/\pi^\alpha$, where $\zeta(2\alpha):=\sum_{m=1}^{\infty}m^{-2\alpha}<\infty$.
\end{lemma}
\begin{proof}
With a suitable auxiliary function $G=G^{[-T,T]}$ periodic on $[-T,T]$ satisfying $\int_{-T}^T  G(x) \rd x=\int_{-T}^T g(x) \rd x$, we consider the following bound:
\begin{align}
     \left|
    \int_{-T} ^T g(x) \rd x
    -
Q_{n,T}^*(g)
  \right|
  \le
       \left|
\int_{-T} ^T G(x) \rd x
    -
Q_{n,T}^*(G)
  \right|
  +
         \left|
Q_{n,T}^*(g-G) 
  \right|
  .\label{eq:[-T,T]-err-decomp}
\end{align}
Choosing
\[
G(x):=g(x)-\sum_{\tau=1}^{\alpha}\frac{B^{[-T,T]}_{\tau}(x)}{\tau!}\left(\int_{-T}^T g^{(\tau)}(s) \rd s\right)\text{ for }x\in [-T-\delta,T+\delta],
\]
for an arbitrarily fixed small $\delta\in (0,1)$ turns out to be convenient, as we now explain. Here, 
$B^{[-T,T]}_{\tau}(x)$ is the scaled Bernoulli polynomial of degree $\tau$ on $[-T,T]$, namely
\[
B^{[-T,T]}_{\tau}(x) = (2T)^{\tau-1} B_{\tau}\left(\frac{x+T}{2T}\right)
\]with $B_{\tau}$ being the standard Bernoulli polynomial of degree $\tau$. 
We have $\int_{-T}^T  G(x) \rd x=\int_{-T}^T g(x) \rd x$ by simply noticing that $\int_0^1 B_\tau(x)\rd x =0$ for $\tau\ge 1$.

The function $G$ is ($\alpha-1$)-times differentiable on 
$(-T-\delta,T+\delta)$ 
with $G^{(\alpha-1)}$ being absolutely continuous on $[-T,T]$. Moreover, we have
\begin{align*}
&\int_{-T}^{T} G^{(\tau)}(x)\rd x =\int_{-T}^{T} g^{(\tau)}(x)\rd x - \left(\int_{-T}^{T} B^{[-T,T]}_{0}(x) \rd x \right)\left(\int_{-T}^{T} g^{(\tau)}(s)\rd s\right)  =0 ,
\end{align*}
for $\tau=1,\ldots,\alpha$, and thus the fundamental theorem of calculus tells us
\begin{align*}
&G^ {(\tau)}(-T)=G^ {(\tau)}(T),\quad \text{for }\  \tau=0,\ldots,\alpha-1.
\end{align*}

These properties of $G$ imply the following two Fourier series representations. 
First, from the periodicity and the absolute continuity of $G$ on $[-T,T]$, 
we have the pointwise-convergent Fourier series expansion
\[
G(x)= \sum_{m\in\ZZ} \widehat{G}(m) \phi^{[-T,T]}_{m}(x), 
\]
where $\phi^{[-T,T]}_{m}(x):=\exp(\frac{2\pi \mathrm{i} m(x+T)}{2T})/\sqrt{2T}$,  $m\in\ZZ$ are the orthonormal Fourier basis on $L^2([-T,T])$ and 
$\widehat{G}(m):=\int_{-T}^T G(x) \exp(\frac{-2\pi \mathrm{i} m(x+T)}{2T})/\sqrt{2T}\, \rd x$, $m\in\ZZ$
 are the Fourier coefficients.
Second, 
from the square integrability of $G^{(\alpha)}$, we have the $L^2$-convergent Fourier series representation

\begin{align*}
    \left(G(x)\right)^{(\alpha)}
    &=\sum_{m\in\ZZ} \widehat{G^{(\alpha)}}(m) \phi^{[-T,T]}_{m}(x)=\sum_{m\in\ZZ} \left(\frac{2\pi \mathrm{i} m}{2T}\right)^{\alpha} \widehat{G}(m) \phi^{[-T,T]}_{m}(x),
\end{align*}
where in the second equality we repeatedly used the integration by parts.

Using these representations, we obtain
\begin{align}
   \biggl|
   Q_{n,T}^*(G)&
   -
    \int_{-T}^T G(x) \rd x
    \biggr|\notag
=
   \biggl|
    \frac{2T}{n} \sum_{j=0}^{n-1}\sum_{m\in\ZZ} \widehat{G}(m) \phi^{[-T,T]}_{m}(\xi^*_j)
    -
    \sqrt{2T}\widehat{G}(0)
\biggr|
\\
&=
 \Biggl|
 \sqrt{2T}\sum_{m\in\ZZ\setminus\{0\}} \widehat{G}(mn) 
 \Biggr|\notag
\\
 &\le
 \sqrt{2T} \Biggl(
 \sum_{m\in\ZZ\setminus\{0\}} |\widehat{G}(mn)|^2 \left(\frac{2\pi m n}{2T}\right)^{2\alpha} 
 \Biggr)^{1/2}
 \Biggl(
 \sum_{m\in\ZZ\setminus\{0\}}  \left(\frac{2T}{2\pi m n}\right)^{2\alpha} 
 \Biggr)^{1/2}\notag
 \\
 &\le
 \sqrt{2T}\|G\|_{\alpha,[-T,T]}
 \sqrt{2\zeta(2\alpha)}
 \left(\frac{T}{\pi n}\right)^{\alpha} 
 =:C_\alpha \|G\|_{\alpha,[-T,T]} T^{\alpha+1/2} \frac{1}{n^\alpha},\label{eq:bd-in-G}
\end{align}
where in the first to second lines we used the pointwise convergence of the series and\begin{align*}
    \frac{2T}{n} \sum_{j=0}^{n-1} \phi^{[-T,T]}_{m}(\xi^*_j)
=
\frac{\sqrt{2T}}{n} \sum_{j=0}^{n-1} \exp(2\pi\mathrm{i}\:\! m j/n)
=
\begin{cases}
\sqrt{2T} & \text{if  } m\equiv 0 \pmod{n},\\
0 & \text{otherwise},
\end{cases}
\end{align*}
while in the fourth line we used the Parseval identity.
The equation \eqref{eq:bd-in-G} is further bounded by $C_\alpha \|g\|_{\alpha,[-T,T]} T^{\alpha+1/2}\frac{1}{n^\alpha}$ since
\begin{align*}
\|G\|_{\alpha,[-T,T]}&=\left(  \left(\int_{-T}^T G(x) \rd x\right)^2 +\int_{-T}^T | \left(G\right)^{(\alpha)}(x) |^2\rd x
\right)^{1/2}
\\
&=
\left(  \left(\int_{-T}^T  g(x) \rd x\right)^2 +\int_{-T}^T  | g^{(\alpha)}(x)|^2 \rd x -\frac{1}{2T}\left(\int_{-T}^T  g^{(\alpha)}(y)\rd y \right)^2
\right)^{1/2}
\\
&\le
\|g\|_{\alpha,[-T,T]}.
\end{align*}
Now we bound the the second term of the right hand side in \eqref{eq:[-T,T]-err-decomp}:
\begin{align*}
    \left|
 Q_{n,T}^*(g-G) 
  \right|
  &=
  \Biggl|\frac{1}{n}\sum_{j=0}^{n-1}\sum_{\tau=1}^{\alpha}\frac{B^{[-T,T]}_{\tau}(\xi^*_j)}{\tau!}
    \biggl(\int_{-T}^T g^{(\tau)}(s) \rd s
    \biggr)\Biggr|
  \\
  &\le
  \sum_{\tau=1}^{\alpha}
    \Bigg|
        \frac{1}{n}\sum_{j=0}^{n-1}\frac{B^{[-T,T]}_{\tau}(\xi^*_j)}{\tau!}
    \Bigg|
    \bigl| g^{(\tau-1)}(T)-g^{(\tau-1)}(-T)\bigr|
  \\
  &\le
   \sum_{\tau=1}^{\alpha}\frac{(2T)^{\tau-1}}{2} (2 \|g\|_{\alpha,\mathrm{decay}})\,\rme^{-(1-\varepsilon)T^2/2}
   \\
   &\le
   \alpha \max\{1,(2T)^{\alpha-1}\}\|g\|_{\alpha,\mathrm{decay}}\,\rme^{-(1-\varepsilon)T^2/2},
\end{align*} 
where in the penultimate line we used $|\frac{B^{[-T,T]}_{\tau}(x)}{\tau!}|\le \frac{(2T)^{\tau-1}}{2}$ for $x\in[-T,T]$; see~\cite[Equation~(6)]{NS2021} or~ \cite{L1940}. 
Together with \eqref{eq:bd-in-G}, the statement follows.
\end{proof}

Now, what remains in the bound~\eqref{eq:error-decomp-g} is the error due to chopping the real line to the interval $[-T,T]$. 
The following result tells us how to choose $T$ to obtain a total error bounded by $\calO(n^{-\alpha})$ up to a logarithmic factor. 
\begin{proposition}
\label{prop:trape}
Let $\alpha\in\NN$. 
Suppose that the function $g^{(\tau)}\colon\mathbb{R}\to\mathbb{R}$ is absolutely continuous on any compact interval for $\tau=0,\dots,\alpha-1$, 
and that $g^{(\alpha)}$ is in $L^2(\R)$.
Suppose further that $g$ satisfies
\begin{align}
\|g\|_{\alpha}^*
:=
\sup_{\substack{I\subset\R\\|I|<\infty}} \|g\|_{\alpha,I}
:=
\sup_{\substack{I\subset\R\\|I|<\infty}}\left( \left(\sum_{\tau=0}^{\alpha-1}  \left(\int_I g^{(\tau)}(x) \rd x\right)^2 +\int_I | g^{(\alpha)}(x) |^2\rd x \right)
\right)^{1/2}
< \infty
\label{eq:def-alpha-star}
\end{align}
and
\begin{align}
 \|g\|_{\alpha,\mathrm{decay}} 
 :=
 \sup_{\substack{x\in\R \\ \tau\in\{0,\ldots,\alpha-1\}}}
 \left| \rme^{(1-\varepsilon)x^2/2} \, g^{(\tau)}(x) \right|
 <
 \infty
 ,\qquad\text{for some $\varepsilon\in(0,1)$}.\label{eq:decay-cond-proptrap}
\end{align}
Then, for any integer $n\ge2$, the error for the $n$-point trapezoidal rule $Q^*_{n,T}$ as in \eqref{eq:def-trap} with the cut-off interval $[-T,T]$ given by
\begin{align}
  T
  &=
  \sqrt{\frac{2}{(1-\varepsilon)}\alpha \ln(n)}
  ,\label{eq:def-T}
\end{align}
can be bounded by
\begin{align}\label{eq:tr-bound}
  \left|
    \int_{\R} g(x) \rd x
    -
Q_{n,T}^*(g)
  \right|
  \le
  C \, \left(\|g\|_{\alpha}^* + \|g\|_{\alpha,\mathrm{decay}}\right) \,
  \frac{(\ln n)^{ (\alpha/2 + 1/4) \, }}{n^\alpha}
  ,
\end{align}
where the constant $C$ is independent of $n$ and $g$ but depends on $\alpha$ and $\varepsilon$.
\end{proposition}
\begin{proof}

Consider the bound \eqref{eq:error-decomp-g}.
The error due to cutting off the integration domain is bounded by 
\begin{align*}
\left|\int_{\R} g(x) \rd x -\int_{-T}^T g(x) \rd x  \right| 
&\le
 2\|g\|_{\alpha,\mathrm{decay}} \int_T^\infty \rme^{-(1-\varepsilon)x^2/2} \rd x
 \\
 &\le
 \frac{2\|g\|_{\alpha,\mathrm{decay}}}{(1-\varepsilon)T}\int_T^\infty (1-\varepsilon) x \rme^{-(1-\varepsilon)x^2/2} \rd x
 \\&=
 \frac{2\|g\|_{\alpha,\mathrm{decay}}}{(1-\varepsilon)T} \rme^{-(1-\varepsilon)T^2/2}
=
 \frac{\sqrt{2}\|g\|_{\alpha,\mathrm{decay}}}{\sqrt{\alpha(1-\varepsilon)}}n^{-\alpha} (\ln(n))^{-1/2}.
\end{align*}
Noting that $n\ge2$ and \eqref{eq:def-T} imply $2T>1$ so that $\max\{1,(2T)^{\alpha-1}\}=(2T)^{\alpha-1}$,  from Lemma~\ref{lem:trap-error-inside} we have
\begin{align*}
\left|
   \int_{\R} g(x) \rd x
    -
   Q_{n,T}^*(g)
\right|
  &\le
    \frac{\sqrt{2}\|g\|_{\alpha,\mathrm{decay}}}{\sqrt{\alpha(1-\varepsilon)}} (\ln(n))^{-1/2}\, n^{-\alpha}
  \\
  &\qquad\quad+
  C_1 \, \|g\|_{\alpha}^*
    \,
    (\ln(n))^{(\alpha/2+1/4)} \,
    n^{-\alpha}
  \\
  &\qquad\quad+
  C_2
    \,
    \|g\|_{\alpha,\mathrm{decay}}
     \, (\ln(n))^{(\alpha/2-1/2) }
    \,
    n^{-\alpha}\\
    &\le
     C \, \left(\|g\|_{\alpha}^* + \|g\|_{\alpha,\mathrm{decay}}\right) \,
  (\ln n)^{ (\alpha/2 + 1/4)} \, n^{-\alpha},
\end{align*}
where the constants $C_1$, $C_2$ and $C$ are independent of $n$ and $g$ but depend on $\alpha$ and $\varepsilon$. Thus the claim is proved.
\end{proof}
\begin{remark}
The result \cite[Theorem~2]{NS2021} by Nuyens and Suzuki obtained for a class of quasi-Monte Carlo methods called good lattice rules can be seen as a multidimensional counterpart of Proposition~\ref{prop:trape}.
Indeed, it can be checked that the trapezoidal rule is a good lattice, and thus under the same assumption as Proposition~\ref{prop:trape}, the result therein is immediately applicable to the trapezoidal rule. However, we obtained a better bound by exploiting our one-dimensional setting in Proposition~\ref{prop:trape}. 
Compare this result with \cite[Theorem~2]{NS2021} with the parameters therein being $d=1$, $\beta=(1-\varepsilon)/2$ and $p=q=2$ to see the improvement.
\end{remark}

Our results offer several insights to interpret results available in  the literature.

\begin{sloppypar}
In the context of spectral methods, Boyd \cite{Boyd.JP_2001_book,Boyd.JP_2009} pointed out that the Gauss--Hermite points are distributed roughly uniformly over the interval $[O(-n^{1/2}),O(n^{1/2})]$, and thus, the total number of point being $n$, the spacing between adjacent points decreases only as $O(n^{-1/2})$; see for example \cite[Chapter~17]{Boyd.JP_2001_book} and \cite[Fig.~6]{Boyd.JP_2009}.
The proof of Theorem~\ref{thm:GH-LB} (see also Corollary~\ref{cor:delta-LB}) shows that it is this  slow decrease of the spacing that causes the sub-optimal convergence rate.
\end{sloppypar}

In \cite[Section 5]{T2021}, Trefethen compared Gausss--Hermite quadrature and various quadrature formulas, including the trapezoidal rule. 
Although the focus there was analytic integrands, 
the author also discusses the nonanalytic case. 
On page 142, he seems to have reasoned that 
\footnote{``The ratio increases to nearly order $n^{1/2}$ for nonanalytic functions $f$, where intervals growing just logarithmically rather than algebraically with $n$ are appropriate for balancing domain-truncation and discretization errors.'' \cite[p.\ 142]{T2021}.}
for the nonanalytic functions on $\mathbb{R}$ decaying at a suitable rate (presumably at the rate  $\exp(-x^2)$ as $x\to\infty$ including derivatives)  
the right choice of the cut-off of the domain that 
balances domain-truncation and quadrature errors
should be logarithmic in $n$,
while the Gauss--Hermite rule distributes quadrature points to unnecessarily wide intervals  $[\mathcal{O}(-n^{1/2})), \mathcal{O}(n^{1/2}))]$. 
Proposition~\ref{prop:trape} supports this point for the trapezoidal rule. 
Indeed, 
under the exponential decay condition \eqref{eq:decay-cond-proptrap}, 
we cut off the integration domain logarithmically \eqref{eq:tr-bound}, and we achieve the optimal rate  $\mathcal{O}(n^{-\alpha})$ up to a logarithmic factor. 
Note, however, that for polynomially decaying finitely smooth functions, we expect the right choice of the domain cut-off to grow algebraically; see \cite[Theorem 2 (ii)]{NS2021} for a related result.

In Proposition~\ref{prop:trape}, the choice of the cut-off interval \eqref{eq:def-T} requires the smoothness parameter $\alpha$, which might not be known in practice. 
Replacing $\alpha$ in \eqref{eq:def-T} with any slowly increasing function $\gamma(n)$, such as $\max\{\ln(\ln(n)),0\}$, yields a less tight bound for the trapezoidal rule, but with an $\alpha$-free construction, still achieving the optimal rate up to a factor of $(\gamma(n)\ln n)^{ (\alpha/2 + 1/4)}$.
\begin{corollary}\label{cor:alpha-free}
Suppose that assumptions in Proposition~\ref{prop:trape} are satisfied. Let $\gamma(n)\colon\mathbb{N}\to[0,\infty)$
be a non-decreasing function satisfying  $\lim_{n\to\infty}\gamma(n)=\infty$.
Then, for any integer $n\geq\gamma^{-1}(\alpha):=\min\{m\in\mathbb{N}\mid\gamma(m)\geq\alpha\}$,
the error for $Q_{n,\tilde{T}}^{*}$ as in~\eqref{eq:def-trap} with 
\[
\tilde{T}=\sqrt{\frac{2}{(1-\varepsilon)}\gamma(n)\ln(n)}
\]
can be bounded by 
\begin{align}\left|
    \int_{\R} g(x) \rd x
    -
Q_{n,\tilde{T}}^*(g)
  \right|
  \le
  C \, \left(\|g\|_{\alpha}^* + \|g\|_{\alpha,\mathrm{decay}}\right) \,
  \frac{(\gamma(n)\ln n)^{ (\alpha/2 + 1/4) \, }}{n^\alpha}
,
\end{align}
where the constant $C$ is independent of $n$ and $g$ but depends on $\alpha$ and $\varepsilon$.
\end{corollary}

Now we are going to show that Proposition~\ref{prop:trape} is applicable to the weighted Sobolev space $\Hscr_{\alpha}$.
In view of the optimal rate in \cite[Theorem 1]{DILP2018} for the Hermite space $\cH^{\mathrm{Hermite}}_{\alpha}$ and the characterisation of $\cH^{\mathrm{Hermite}}_{\alpha}$ with $\Hscr_{\alpha}$ discussed in Section~\ref{sec:space}, the resulting rate below establishes the optimality, up to a logarithmic factor, of our trapezoidal rule.
\begin{theorem}\label{thm:trape-opt}
Fix $\varepsilon\in(1/2,1)$ arbitrarily.
For $f\in\Hscr_{\alpha}$ with $\alpha\in\NN$, consider $Q^*_{n,T}(f\rho)$ as in \eqref{eq:def-trap} with
  $T
  =
  \sqrt{\frac{2}{1-\varepsilon}\alpha \ln(n)}$. Then, we have
\begin{align*}
  |
    I(f)
    -
    Q_{n,T}^*(f\rho)
  |
  \le
  C  \|f\|_{\alpha}  
  \frac{(\ln n)^{ (\alpha/2 + 1/4) \, }}{n^\alpha}
\end{align*}
for any integer $n\ge 2$, 
where the constant $C$ is independent of $n$ and $f$ but depends on $\alpha$ and $\varepsilon$.
\end{theorem}
\begin{proof}
Let $g:=f\rho$. In view of Proposition~\ref{prop:trape}, it suffices to show $\|g\|_{\alpha}^* + \|g\|_{\alpha,\mathrm{decay}}\leq C\|f\|_\alpha$ for some constant $C>0$, where $\|g\|_{\alpha}^*$ and $\|g\|_{\alpha,\mathrm{decay}}$ are as in 
 \eqref{eq:def-alpha-star} and \eqref{eq:decay-cond-proptrap} with $\varepsilon\in(1/2,1)$, respectively. 

We first show $\|g\|_{\alpha}^*\leq C \|f\|_{\alpha}$.
We have
\[
(\|g\|^*_\alpha)^2 
\le
\sum_{\tau=0}^{\alpha-1}  \Bigl(\int_\R |g^{(\tau)}(x)| \rd x\Bigr)^2 +\int_\R | g^{(\alpha)}(x) |^2\rd x,
\]
but using \eqref{eq:Hermite-def} and the chain rule, 
for $\tau=0,\ldots,\alpha-1$ we have
\begin{align*}
\|g^{(\tau)}&\|_{L^1(\R)}
\le
\sum_{\ell=0}^\tau \binom{\tau}{\ell} \| f^{(\tau-\ell)}(x)\rho^{(\ell)}(x) \|_{L^1(\R)}
\\
&=
\sum_{\ell=0}^\tau \binom{\tau}{\ell} \left(\int_{\R} \left|\rho(x) f^{(\tau-\ell)}(x)\sqrt{\ell!}(-1)^{\ell} H_{\ell}(x) \right| \rd x\right)
    \\
&\le
    \sum_{\ell=0}^\tau \binom{\tau}{\ell} \sqrt{\ell!} \left(\int_{\R} \left| f^{(\tau-\ell)}(x)\right|^2\rho(x)\rd x\right)^{1/2}\left(\int_{\R} \left|H_{\ell}(x) \right|^2 \rho(x) \rd x\right)^{1/2}
    <\infty,
\end{align*}
while for $\tau=\alpha$ we have
\begin{align*}
    \| g^{(\alpha)}(x) \|_{L^2(\R)}
    &\le
    \sum_{\ell=0}^\alpha \binom{\alpha}{\ell} \| f^{(\alpha-\ell)}(x)\rho^{(\ell)}(x) \|_{L^2(\R)}
    \\
    &=
    \sum_{\ell=0}^\alpha \binom{\alpha}{\ell} \left(\int_{\R} \left|\rho(x) f^{(\alpha-\ell)}(x)\sqrt{\ell!}(-1)^{\ell} H_{\ell}(x) \right|^2 \rd x\right)^{1/2}
    \\
    &\le
    \sum_{\ell=0}^\alpha \binom{\alpha}{\ell} \left( \sup_{t\in\R} \left(H_{\ell}^2(t)\rho(t) \right) \int_{\R} \rho(x) \bigl|f^{(\alpha-\ell)}(x)\bigr|^2\ell!    \rd x\right)^{1/2}
    <\infty.
\end{align*}
Hence, $\|g\|^*_\alpha\leq C \|f\|_{\alpha}$ holds.

To show $\|g\|_{\alpha,\mathrm{decay}}\leq C \|f\|_{\alpha}$, let $h_{\tau}:=\rho^{\varepsilon-1} g^{(\tau)}$ for $\tau=0,\ldots,\alpha-1$. Then, we have
\begin{align*}
    \| h_{\tau} \|_{L^2(\R)}
    &\le
    \sum_{\ell=0}^\tau \binom{\tau}{\ell} \|\rho^{\varepsilon-1}(x) f^{(\tau-\ell)}(x)\rho^{(\ell)}(x) \|_{L^2(\R)}
    \\
    &=
    \sum_{\ell=0}^\tau \binom{\tau}{\ell} \left(\int_{\R} \left|\rho^{\varepsilon}(x) f^{(\tau-\ell)}(x)\sqrt{\ell!}(-1)^{\ell} H_{\ell}(x) \right|^2 \rd x\right)^{1/2}
    \\
    &\le
    \sum_{\ell=0}^\tau \binom{\tau}{\ell} \left( \sup_{t\in\R} \left(H_{\ell}^2(t)\rho^{2\varepsilon-1}(t) \right) \int_{\R} \rho(x) \left|f^{(\tau-\ell)}(x)\right|^2 \ell!   \rd x\right)^{1/2} 
    <\infty
\end{align*}
and 
\begin{align*}
\|&h_{\tau}'(x)\|_{L^2(\R)}
\le
\|(1-\varepsilon)x g^{(\tau)}(x)\rho^{\varepsilon-1} (x) \|_{L^2(\R)} + \|g^{(\tau+1)}(x)\rho^{\varepsilon-1}(x)\|_{L^2(\R)}
\\
&
\le
\sum_{\ell=0}^\tau \binom{\tau}{\ell}\left( \int_{\R} \left|(1-\varepsilon)x\rho^{\varepsilon}(x) f^{(\tau-\ell)}(x)\sqrt{\ell!}(-1)^{\ell} H_{\ell}(x) \right|^2 \rd x\right)^{1/2}
\\
&\quad+
\sum_{\ell=0}^{\tau+1} \binom{\tau+1}{\ell}\left(\int_{\R} \left|\rho^{\varepsilon}(x) f^{({\tau+1}-\ell)}(x)\sqrt{\ell!}(-1)^{\ell} H_{\ell}(x) \right|^2 \rd x\right)^{1/2} \\
&\le
\sum_{\ell=0}^\tau \binom{\tau}{\ell} \left( \sup_{t\in \R}\left|\rho^{2\varepsilon-1}(t)(1-\varepsilon)^2 t^2 H^2_{\ell}(t) \right| \int_{\R} |f^{(\tau-\ell)}(x)|^2\rho(x)\ell! \rd x\right)^{1/2}\\
&\quad+
\sum_{\ell=0}^{\tau+1} \binom{\tau+1}{\ell}\left(\sup_{t\in \R}\left| \rho^{2\varepsilon-1}(t) H_{\ell}^2(t) \right|  \int_{\R} |f^{({\tau+1}-\ell)}(x)|^2 \ell! \rho(x)\rd x\right)^{1/2}<
\infty.
\end{align*}
Thus, from the Sobolev inequality, e.g., \cite[Theorem~8.8]{B2011}, for $\tau=0,\dots,\alpha-1$ we have $\|h_{\tau}\|_\infty\le C\|h_{\tau}\|_{W^{1,2}(\R)}<\infty$. This completes the proof.
\end{proof}
Theorem~\ref{thm:trape-opt} is an application of Proposition~\ref{prop:trape} to $g=f\rho$ with $f\in \Hscr_{\alpha}$.  
Similarly, applying Corollary~\ref{cor:alpha-free} to $g=f\rho$ yields a trapezoidal rule whose construction is independent of $\alpha$ with the optimal convergence rate up to a factor of $(\gamma(n)\ln n)^{ (\alpha/2 + 1/4)}$ for functions in $\Hscr_{\alpha}$. 
Since the argument is straightforward from 
Corollary~\ref{cor:alpha-free} and 
Theorem~\ref{thm:trape-opt}, we omit the details.

\subsection*{Details of Figure~\ref{fig:GH-sub}}
Now we are ready to discuss the details of Figure~\ref{fig:GH-sub} in Section~\ref{sec:intro}. 
The trapezoidal rule used there is $Q^*_{n,T}$ as in \eqref{eq:def-trap} with $T=\sqrt{\frac{2}{1-\varepsilon}\alpha\ln (n)}$ and $\varepsilon=0.51$. 
Here, we chose  $\alpha=p$, since $f(x)=|x|^p$ is in $\Hscr_p$ but not in $\Hscr_{p+1}$.
The number of points $n$ is chosen to be odd for the trapezoidal rule and even for Gauss--Hermite quadrature, so that both quadrature rules do not evaluate at the origin $x=0$ where the integrand is not smooth. 

The rate around $\calO(n^{-p/2-0.5})$ we observe for Gauss--Hermite quadrature is consistent with the matching bounds of the of order $n^{-\alpha/2}=n^{-p/2}$ in the sense of worst-case error, since $f(x)=|x|^p$ is a specific element from $\Hscr_p$; 
the rate around $\calO(n^{-p-0.8})$ we observe for the trapezoidal rule also supports our results,  according to which we expect to see at least $\calO(n^{-p})$ for any function in $\Hscr_p$.

\section{Conclusions}\label{sec:conc}
In this paper, we proved the sub-optimality of Gauss--Hermite quadrature and the optimality of the trapezoidal rule for functions with finite smoothness, in the sense of worst-case error.
The lower bound presented for Gauss--Hermite quadrature is sharp, and the upper bound presented for the trapezoidal rule is also sharp, up to a logarithmic factor.

To establish the lower bound for Gauss--Hermite rule, we constructed a sequence of fooling functions. 
This strategy also demonstrated that what causes this lower bound is the placement of quadrature points, and thus tuning the quadrature weights does not improve the bound. 

A key for showing the optimality of the trapezoidal rule was the auxiliary periodic function in Lemma~\ref{lem:trap-error-inside}. 
The function used there is in fact an orthogonal projection in a suitable sense;  
for details, we refer to \cite{NS2021}. 
Needless to say, upon the domain truncation $[-T,T]$, other quadrature rules on the finite interval, such as Clenshaw--Curtis or Gauss--Legendre quadratures, can also be used. 
For these quadrature rules, analogous upper bounds should be able to be derived, without the necessity of introducing the aforementioned periodic function. 
Since these quadratures are arguably more complicated to use than the trapezoidal rule, and the error analysis should be less involved, we had left them out from the scope of this paper.

Our results suggest that the truncated trapezoidal rule may be also promising for high-dimensional problems. 
One generalisation of the trapezoidal rule to the multidimensional setting is the lattice rule. 
Nuyens and Suzuki \cite{NS2021} studied this method for the integration problem on $\mathbb{R}^d$ with respect to the Lebesgue measure. 
To verify if it works well for the Gaussian measure in a high-dimensional setting is kept for future works.

Another generalisation to high-dimensional settings is by the Smolyak-type algorithms. 
As mentioned in Section~\ref{sec:intro}, this type of methods based on Gauss--Hermite points is widely used. 
In light of the results presented in this paper, especially when the target function is expected to have limited smoothness, the trapezoidal rule may be a better choice.
Investigating these speculations is also kept for future works.

\section*{Acknowledgments}
Part of this work was carried out when Yoshihito Kazashi was working at CSQI, Institute of Mathematics, \'Ecole Polytechnique F\'ed\'erale de Lausanne, Switzerland. We thank Dirk Nuyens and Ken’ichiro Tanaka for their valuable comments.

\bibliographystyle{siamplain}
\bibliography{gh-non-optimal}

\begin{thebibliography}{10}

\bibitem{BNT2010SIREV}
{\sc I.~Babu\v{s}ka, F.~Nobile, and R.~Tempone}, {\em A {{Stochastic
  Collocation Method}} for {{Elliptic Partial Differential Equations}} with
  {{Random Input Data}}}, SIAM Rev., 52 (2010), pp.~317--355,
  \url{https://doi.org/10.1137/100786356}.

\bibitem{Barrett.W_1961}
{\sc W.~Barrett}, {\em Convergence properties of {Gaussian} quadrature
  formulae}, Comput. J., 3 (1961), pp.~272--273,
  \url{https://doi.org/10.1093/comjnl/3.4.272}.

\bibitem{B1998}
{\sc V.~I. Bogachev}, {\em Gaussian measures}, vol.~62 of Mathematical Surveys
  and Monographs, American Mathematical Society, Providence, RI, 1998,
  \url{https://doi.org/10.1090/surv/062}.

\bibitem{Boyd.JP_2001_book}
{\sc J.~P. Boyd}, {\em Chebyshev and {F}ourier spectral methods}, Dover
  Publications, Inc., Mineola, NY, second~ed., 2001.

\bibitem{Boyd.JP_2009}
{\sc J.~P. Boyd}, {\em Large-degree asymptotics and exponential asymptotics for
  {F}ourier, {C}hebyshev and {H}ermite coefficients and {F}ourier transforms},
  J. Engrg. Math., 63 (2009), pp.~355--399,
  \url{https://doi.org/10.1007/s10665-008-9241-3},
  \url{https://doi.org/10.1007/s10665-008-9241-3}.

\bibitem{Brandimarte2006book}
{\sc P.~Brandimarte}, {\em Numerical {{Methods}} in {{Finance}} and
  {{Economics}}: {{A MATLAB}}\textregistered -{{Based Introduction}}}, {John
  Wiley \& Sons, Inc.}, 2006, \url{https://doi.org/10.1002/0470080493}.

\bibitem{braun2021satellite}
{\sc T.~M. Braun and W.~R. Braun}, {\em Satellite {{Communications Payload}}
  and {{System}}}, {Wiley}, 2nd~ed., 2021,
  \url{https://doi.org/10.1002/9781119384342}.

\bibitem{BBM2020book}
{\sc L.~Brevault, M.~Balesdent, and J.~Morio}, {\em Aerospace system analysis
  and optimization in uncertainty}, vol.~156 of Springer Optimization and Its
  Applications, Springer, Cham, 2020,
  \url{https://doi.org/10.1007/978-3-030-39126-3}.

\bibitem{B2011}
{\sc H.~Brezis}, {\em Functional analysis, {S}obolev spaces and partial
  differential equations}, Universitext, Springer, New York, 2011,
  \url{https://doi.org/10.1007/978-0-387-70914-7}.

\bibitem{Canuto.C_etal_2006_book}
{\sc C.~Canuto, M.~Hussaini, A.~Quarteroni, and T.~Zang}, {\em Spectral methods
  fundamentals in single domains}, {Springer}, 2006,
  \url{https://doi.org/10.1007/978-3-540-30726-6}.

\bibitem{C2018}
{\sc P.~Chen}, {\em Sparse quadrature for high-dimensional integration with
  {G}aussian measure}, ESAIM Math. Model. Numer. Anal., 52 (2018),
  pp.~631--657, \url{https://doi.org/10.1051/m2an/2018012}.

\bibitem{DR1984}
{\sc P.~J. Davis and P.~Rabinowitz}, {\em Methods of numerical integration},
  Computer Science and Applied Mathematics, Academic Press, Inc., Orlando, FL,
  2nd~ed., 1984.

\bibitem{DM2003}
{\sc B.~Della~Vecchia and G.~Mastroianni}, {\em Gaussian rules on unbounded
  intervals}, J. Complexity, 19 (2003), pp.~247--258,
  \url{https://doi.org/10.1016/S0885-064X(03)00008-6}.

\bibitem{DHP2015}
{\sc J.~Dick, A.~Hinrichs, and F.~Pillichshammer}, {\em Proof techniques in
  quasi-{M}onte {C}arlo theory}, J. Complexity, 31 (2015), pp.~327--371,
  \url{https://doi.org/10.1016/j.jco.2014.09.003}.

\bibitem{DILP2018}
{\sc J.~Dick, C.~Irrgeher, G.~Leobacher, and F.~Pillichshammer}, {\em On the
  optimal order of integration in {H}ermite spaces with finite smoothness},
  SIAM J. Numer. Anal., 56 (2018), pp.~684--707,
  \url{https://doi.org/10.1137/16M1087461}.

\bibitem{Dick.J_Kuo_Sloan_2013_ActaNumerica}
{\sc J.~Dick, F.~Y. Kuo, and I.~H. Sloan}, {\em High-dimensional integration:
  {{The}} quasi-{{Monte}} {C}arlo way}, Acta Numer., 22 (2013), pp.~133--288,
  \url{https://doi.org/10.1017/S0962492913000044}.

\bibitem{D2021}
{\sc D.~D{\~u}ng}, {\em Sparse-grid polynomial interpolation approximation and
  integration for parametric and stochastic elliptic {{PDEs}} with lognormal
  inputs}, ESAIM Math. Model. Numer. Anal., 55 (2021), pp.~1163--1198,
  \url{https://doi.org/10.1051/m2an/2021017}.

\bibitem{dung2022analyticity}
{\sc D.~D{\~u}ng, V.~K. Nguyen, C.~Schwab, and J.~Zech}, {\em Analyticity and
  sparsity in uncertainty quantification for {P}{D}{E}s with {G}aussian random
  field inputs}, arXiv preprint {\tt arXiv:2201.01912 [math.NA]},  (2022).

\bibitem{Freud1972}
{\sc G.~Freud}, {\em A contribution to the problem of weighted polynomial
  approximation}, in Linear operators and approximation ({P}roc. {C}onf.,
  {O}berwolfach, 1971), 1972, pp.~431--447. Internat. Ser. Numer. Math., Vol.
  20.

\bibitem{FR2008book}
{\sc G.~Fusai and A.~Roncoroni}, {\em Implementing models in quantitative
  finance: methods and cases}, Springer Finance, Springer, Berlin, 2008,
  \url{https://doi.org/10.1007/978-3-540-49959-6}.

\bibitem{Gezerlis2020book}
{\sc A.~Gezerlis}, {\em Numerical Methods in Physics with Python}, Cambridge
  University Press, 2020, \url{https://doi.org/10.1017/9781108772310}.

\bibitem{GilEtAl2007}
{\sc A.~Gil, J.~Segura, and N.~M. Temme}, {\em Numerical {{Methods}} for
  {{Special Functions}}}, {Society for Industrial and Applied Mathematics},
  2007, \url{https://doi.org/10.1137/1.9780898717822}.

\bibitem{gnewuch2021countable}
{\sc M.~Gnewuch, M.~Hefter, A.~Hinrichs, and K.~Ritter}, {\em Countable tensor
  products of {H}ermite spaces and spaces of {G}aussian kernels}, J.
  Complexity, 71 (2022), \url{https://doi.org/10.1016/j.jco.2022.101654}.
\newblock Paper No. 101654, 40.

\bibitem{Goodwin1949}
{\sc E.~T. Goodwin}, {\em The evaluation of integrals of the form
  $\int_{-\infty}^{\infty} f(x) e^{-x^2} dx$}, Proc. Cambridge Philos. Soc., 45
  (1949), pp.~241--245, \url{https://doi.org/10.1017/s0305004100024786}.

\bibitem{Graham.I_eta_2015_Numerische}
{\sc I.~G. Graham, F.~Y. Kuo, J.~A. Nichols, R.~Scheichl, C.~Schwab, and I.~H.
  Sloan}, {\em Quasi-{{Monte Carlo}} finite element methods for elliptic
  {{PDEs}} with lognormal random coefficients}, Numer. Math., 131 (2015),
  pp.~329--368, \url{https://doi.org/10.1007/s00211-014-0689-y}.

\bibitem{GH2010}
{\sc M.~Griebel and M.~Holtz}, {\em Dimension-wise integration of
  high-dimensional functions with applications to finance}, J. Complexity, 26
  (2010), pp.~455--489, \url{https://doi.org/10.1016/j.jco.2010.06.001}.

\bibitem{Guo.BY_1999_Hermite}
{\sc B.-Y. Guo}, {\em Error estimation of {{Hermite}} spectral method for
  nonlinear partial differential equations}, Math. Comp., 68 (1999),
  pp.~1067--1078, \url{https://doi.org/10.1090/S0025-5718-99-01059-5}.

\bibitem{HS2019ML}
{\sc L.~Herrmann and C.~Schwab}, {\em Multilevel quasi-{M}onte {C}arlo
  integration with product weights for elliptic {PDE}s with lognormal
  coefficients}, ESAIM Math. Model. Numer. Anal., 53 (2019), pp.~1507--1552,
  \url{https://doi.org/10.1051/m2an/2019016}.

\bibitem{IG2015_Hermite}
{\sc C.~Irrgeher and G.~Leobacher}, {\em High-dimensional integration on
  $\mathbb{R}^d$, weighted {{Hermite}} spaces, and orthogonal transforms}, J.
  Complexity, 31 (2015), pp.~174--205,
  \url{https://doi.org/10.1016/j.jco.2014.09.002}.

\bibitem{YK2019}
{\sc Y.~Kazashi}, {\em Quasi\textendash{{Monte Carlo}} integration with product
  weights for elliptic {{PDEs}} with log-normal coefficients}, IMA J. Numer.
  Anal., 39 (2019), pp.~1563--1593,
  \url{https://doi.org/10.1093/imanum/dry028}.

\bibitem{L1940}
{\sc D.~H. Lehmer}, {\em On the maxima and minima of {B}ernoulli polynomials},
  Amer. Math. Monthly, 47 (1940), pp.~533--538,
  \url{https://doi.org/10.2307/2303833}.

\bibitem{mao2017hermite}
{\sc Z.~Mao and J.~Shen}, {\em Hermite spectral methods for fractional {{PDEs}}
  in unbounded domains}, SIAM J. Sci. Comput., 39 (2017), pp.~A1928--A1950,
  \url{https://doi.org/10.1137/16M1097109}.

\bibitem{MM1994}
{\sc G.~Mastroianni and G.~Monegato}, {\em Error estimates for
  {G}auss-{L}aguerre and {G}auss-{H}ermite quadrature formulas}, in
  Approximation and computation ({W}est {L}afayette, {IN}, 1993), vol.~119 of
  Internat. Ser. Numer. Math., Birkh\"{a}user Boston, Boston, MA, 1994,
  pp.~421--434.

\bibitem{MS2001}
{\sc M.~Mori and M.~Sugihara}, {\em The double-exponential transformation in
  numerical analysis}, J. Comput. Appl. Math., 127 (2001), pp.~287--296,
  \url{https://doi.org/10.1016/S0377-0427(00)00501-X}.

\bibitem{MNT2013_StochasticCollocationMethod}
{\sc M.~Motamed, F.~Nobile, and R.~Tempone}, {\em A stochastic collocation
  method for the second order wave equation with a discontinuous random speed},
  Numer. Math., 123 (2013), pp.~493--536,
  \url{https://doi.org/10.1007/s00211-012-0493-5}.

\bibitem{MNT2015_AnalysisComputationElasticWave}
{\sc M.~Motamed, F.~Nobile, and R.~Tempone}, {\em Analysis and computation of
  the elastic wave equation with random coefficients}, Comput. Math. Appl., 70
  (2015), pp.~2454--2473, \url{https://doi.org/10.1016/j.camwa.2015.09.013}.

\bibitem{NT2009parabolic}
{\sc F.~Nobile and R.~Tempone}, {\em Analysis and implementation issues for the
  numerical approximation of parabolic equations with random coefficients},
  Int. J. Numer. Meth. Engng, 80 (2009), pp.~979--1006,
  \url{https://doi.org/10.1002/nme.2656}.

\bibitem{NS2021}
{\sc D.~Nuyens and Y.~Suzuki}, {\em Scaled lattice rules for integration on
  $\mathbb{R}^d$ achieving higher-order convergence with error analysis in
  terms of orthogonal projections onto periodic spaces}, Math. Comp., 92
  (2023), pp.~307--347, \url{https://doi.org/10.1090/mcom/3754}.

\bibitem{Shen.J_2011_book}
{\sc J.~Shen, T.~Tang, and L.~Wang}, {\em Spectral methods: {{Algorithms}},
  {{Analysis}} and {{Applications}}}, {Springer}, 2011,
  \url{https://doi.org/10.1007/978-3-540-71041-7}.

\bibitem{SSO1983}
{\sc W.~E. Smith, I.~H. Sloan, and A.~H. Opie}, {\em Product integration over
  infinite intervals. {I}. {R}ules based on the zeros of {H}ermite
  polynomials}, Math. Comp., 40 (1983), pp.~519--535,
  \url{https://doi.org/10.2307/2007528}.

\bibitem{SS2016book}
{\sc B.~A. Stickler and E.~Schachinger}, {\em Basic concepts in computational
  physics}, Springer, Cham, 2~ed., 2016,
  \url{https://doi.org/10.1007/978-3-319-27265-8}.

\bibitem{S1997}
{\sc M.~Sugihara}, {\em Optimality of the double exponential formula\textemdash
  functional analysis approach\textemdash}, Numer. Math., 75 (1997),
  pp.~379--395, \url{https://doi.org/10.1007/s002110050244}.

\bibitem{Szego1975}
{\sc G.~Szeg\"{o}}, {\em Orthogonal polynomials}, vol.~23 of Colloquium
  Publications, American Mathematical Society, 4th~ed., 1975.

\bibitem{T2021}
{\sc L.~N. Trefethen}, {\em Exactness of quadrature formulas}, SIAM Rev., 64
  (2022), pp.~132--150.

\bibitem{TW2014}
{\sc L.~N. Trefethen and J.~A.~C. Weideman}, {\em The {{exponentially
  convergent trapezoidal rule}}}, SIAM Rev., 56 (2014), pp.~385--458,
  \url{https://doi.org/10.1137/130932132}.

\bibitem{Waldvogel2011}
{\sc J.~Waldvogel}, {\em Towards a general error theory of the trapezoidal
  rule}, in Approximation and computation, vol.~42 of Springer Optim. Appl.,
  Springer, New York, 2011, pp.~267--282,
  \url{https://doi.org/10.1007/978-1-4419-6594-3\_17}.

\end{thebibliography}
\end{document}